\documentclass[11pt]{amsart}
\pdfoutput=1

\usepackage{amssymb}
\usepackage{amsthm}
\usepackage{amsfonts}
\usepackage{amsmath}
\usepackage{pmboxdraw}
\usepackage{verbatim} 
\usepackage{graphicx}
\usepackage{color}
\usepackage[colorlinks=true, citecolor=blue, filecolor=black, linkcolor=black, urlcolor=black]{hyperref}
\usepackage{cite}
\usepackage[normalem]{ulem}
\usepackage{subcaption}
\usepackage{bbm}

\usepackage{mathtools}
\mathtoolsset{showonlyrefs}

\newcommand{\RN}[1]{%
	\textup{\uppercase\expandafter{\romannumeral#1}}%
}

\def\wh{\widehat}

\def\Re{ \mathrm{Re}}

\def\C{\mathbb{C}}

\def\R{\mathbb{R}}

\newcommand{\erfc}{\operatorname{erfc}}
\newcommand{\erf}{\operatorname{erf}}

\theoremstyle{plain}
\newtheorem*{thm*}{Theorem}
\newtheorem{thm}{Theorem}[section]
\newtheorem{lem}[thm]{Lemma}
\newtheorem{cor}[thm]{Corollary}
\newtheorem{prop}[thm]{Proposition}
\newtheorem*{prop*}{Proposition}
\newtheorem*{lem*}{Lemma}

\theoremstyle{definition}
\newtheorem*{eg*}{Example}
\newtheorem*{egs*}{Examples}
\newtheorem*{def*}{Definition}
\newtheorem*{Q*}{Question}

\theoremstyle{remark}
\newtheorem{rmk}[thm]{Remark}
\newtheorem*{rmk*}{Remark}
\newtheorem*{rmks*}{Remarks}

\numberwithin{equation}{section}

\makeatletter
\@namedef{subjclassname@2020}{\textup{2020} Mathematics Subject Classification}
\makeatother

\begin{document}
\title[Elliptic random matrices]{Real eigenvalues of elliptic random matrices
}

\author{Sung-soo Byun}
\address{Department of Mathematical Sciences, Seoul National University, Seoul, 08826, Republic of Korea}
\email{sungsoobyun@snu.ac.kr}

\author{Nam-Gyu Kang}
\address{School of Mathematics, Korea Institute for Advanced Study, Seoul, 02455, Republic of Korea }
\email{namgyu@kias.re.kr}

\author{Ji Oon Lee}
\address{Department of Mathematical Sciences, Korea Advanced Institute of Science and Technology, Daejeon, 34141, and School of Mathematics, Korea Institute for Advanced Study, Seoul, 02455, Republic of Korea}
\email{jioon.lee@kaist.edu}

\author{Jinyeop Lee}
\address{Mathematics Institute, Ludwig-Maximilian University of Munich, Theresienstrasse 39, Munich, 80333, Germany}
\email{lee@math.lmu.de}



\begin{abstract}
We consider the real eigenvalues of an $(N \times N)$ real elliptic Ginibre matrix whose entries are correlated through a non-Hermiticity parameter $\tau_N\in [0,1]$. 
In the almost-Hermitian regime where $1-\tau_N=\Theta(N^{-1})$, we obtain the large-$N$ expansion of the mean and the variance of the number of the real eigenvalues. 
Furthermore, we derive the limiting empirical distributions of the real eigenvalues, which interpolate the Wigner semicircle law
and the uniform distribution, the restriction of the elliptic law on the real axis. 
Our proofs are based on the skew-orthogonal polynomial representation of the correlation kernel due to Forrester and Nagao. 
\end{abstract}

\thanks{
The authors are grateful for the support by the National Research Foundation of Korea (NRF-2019R1A5A1028324) and by Samsung Science and Technology Foundation (SSTF-BA1401-51). 
Sung-Soo Byun was partially supported by the German Research Foundation-National Research Foundation of Korea (IRTG 2235). Nam-Gyu Kang was partially supported by a KIAS Individual Grant (MG058103) at Korea Institute for Advanced Study.
Jinyeop Lee was partially supported by the National Research Foundation of Korea (NRF-2020R1F1A1A01070580). 
}


\maketitle


\section{Introduction}

How many eigenvalues of a real random matrix are real? At one extreme, there are real symmetric matrices whose all eigenvalues are real. Without symmetry, Hermiticity fails, and not all eigenvalues are real. Nevertheless, intriguingly, several non-Hermitian random matrices with real entries have real eigenvalues with non-zero probability, contrary to their complex or quaternion counterparts. One of the most fundamental examples of such matrices is a real Ginibre ensemble, a real ($N \times N$) matrix whose entries are i.i.d. Gaussian, for which it was shown in the pioneering work of Edelman, Kostlan, and Shub \cite{MR1231689} that the number of real eigenvalues is of order $\sqrt{N}$.

From the examples of real symmetric and real i.i.d. matrices, it can be heuristically conjectured that as a real random matrix becomes more symmetric, it gets more real eigenvalues. The statement can be made rigorous by considering the following interpolation of the two different random matrices:
\begin{equation}\label{X GOE antiGOE}
X := \sqrt{\frac{1+\tau_N}{2}}\, S + \sqrt{\frac{1-\tau_N}{2}} \, A
\end{equation}
for a (possibly $N$-dependent) parameter $\tau_N$, where $S$ and $A$ are elements of Gaussian orthogonal ensemble (GOE) and anti-symmetric GOE, respectively. Note that $X$ is symmetric if $\tau_N=1$ and $X$ is a Ginibre ensemble if $\tau_N=0$, and thus the (non)-Hermiticity of $X$ is expressed by $\tau_N$. The model is known as a real elliptic Ginibre matrix, which provides a natural bridge between Hermitian and non-Hermitian random matrix theories. The expected number of real eigenvalues $E_{N,\tau_N}$ of $X$ for $\tau_N\equiv \tau \in [0, 1)$ was obtained by Forrester and Nagao \cite{MR2430570}, which is given by
\begin{equation}\label{EN tau<1}
	E_{N,\tau}= \Big(\frac{2}{\pi} \frac{1+\tau}{1-\tau} N \Big)^{\frac12} (1+o(1)).
\end{equation}
It in particular shows that $E_{N,\tau}$ is an increasing function of $\tau$ for any sufficiently large $N$.

The symmetry affects not only the number of real eigenvalues but also its distribution. While the limiting empirical distribution of real eigenvalues of $X$ is the uniform distribution on $(-1-\tau, 1+\tau)$ for $\tau \in [0, 1)$, i.e.,
\begin{equation} \label{rhos}
\rho_\tau^s(x):=\mathbbm{1}_{ [-1-\tau,1+\tau] }(x) \cdot \frac{1}{2(1+\tau)},
\end{equation}
in the limit $\tau \to 1$, the uniform distribution in \eqref{rhos} does not recover the semicircle distribution, which is obviously the limiting distribution in case $\tau=1$. It suggests that a non-trivial transition happens when $\tau$ is close to $1$. 

One can expect that such a transition may appear when the number of the real eigenvalues is of order $N$, in which case thus the distribution of real eigenvalues is affected by the semicircle distribution. It is therefore natural to consider the case for which the right side of \eqref{EN tau<1} is of order $N$, or equivalently $1-\tau_N = \Theta(N^{-1})$, which we will call the almost-Hermitian regime. We consider the following questions in this regime:

\begin{itemize}
 \item What is the number of real eigenvalues?
 \item What is the limiting empirical distribution of the real eigenvalues?
\end{itemize}

In this paper, we aim to answer these questions.

\subsection{Main contributions}

Denote by $\mathcal{N}_{\tau_N}$ the number of real eigenvalues of $X$ in \eqref{X GOE antiGOE}. Our main contributions are as follows: For $N$ even and 
\begin{equation} \label{tau wH}
	\tau_N=1-\tfrac{\alpha^2}{N}, \qquad (\alpha>0),
\end{equation}
we prove
\begin{itemize}
 \item asymptotic formulas for the mean and the variance of $\mathcal{N}_{\tau_N}$, and
 \item a formula for the limiting empirical distribution of real eigenvalues of $X$.
\end{itemize}
As a corollary, we also prove the convergence of $\mathcal{N}_{\tau_N}/N$ in probability.

In Theorem \ref{Thm_EN}, we prove that
\begin{equation}
 E_{N,\tau_N} = c(\alpha) N + O(1)
\end{equation}
with
\begin{equation} \label{c alternative}
	c(\alpha)=\frac{2}{\alpha \sqrt{\pi}} \int_{0}^{1} \erf( \alpha \sqrt{1-s^2} ) \,ds, 
\end{equation}
where $\erf$ is the error function. From the asymptotic behavior of the error function,
\begin{equation} \label{eq:asymp_erf}
\erf(x) \sim 
\begin{cases}
	\frac{2x}{\sqrt{\pi} } &\text{as} \quad x \to 0\,,
	\\
	1 & \text{as} \quad x \to \infty\,,
\end{cases}
\end{equation}
we find that $c(\alpha) \sim 1$ as $\alpha \to 0$ and $c(\alpha) \sim \frac{2}{\alpha \sqrt{\pi}}$ as $\alpha \to \infty$. The former limit corresponds to the Hermitian case, and the latter matches \eqref{EN tau<1} since
\[
 \frac{1}{N} \Big(\frac{2}{\pi} \cdot \frac{1+\tau_N}{1-\tau_N} N \Big)^{\frac12} = \frac{1}{N} \Big(\frac{2}{\pi} \cdot \frac{2-(\alpha^2/N)}{(\alpha^2/N)} N \Big)^{\frac12} \to \frac{2}{\alpha \sqrt{\pi}}
\]
as $\alpha \to \infty$. It in particular shows that the regime $1-\tau_N = \Theta(N^{-1})$ is indeed where the transition for the real eigenvalues happens and our result connects the Hermitian case ($\tau=1$) and the elliptic case ($\tau \in [0, 1)$).

For the variance $V_{N,\tau_N}$ of $\mathcal{N}_{\tau_N}$, in our second main result, Theorem~\ref{Thm_Variance}, we prove that
\begin{equation}
	\lim_{ N \to \infty }\frac{V_{N,\tau_N}}{E_{N,\tau_N} } =r(\alpha)
\end{equation}
for some $r(\alpha)$ satisfying
\begin{equation} \label{r limits}
	\lim_{\alpha\to 0} r(\alpha) = 0, \qquad \lim_{\alpha\to \infty} r(\alpha) = 2-\sqrt{2}.
\end{equation}
Again, our result on the variance connects the Hermitian case and the elliptic case as for a fixed $\tau_N \equiv \tau \in [0, 1)$ it was obtained in \cite{forrester2007eigenvalue,MR2430570} that
\begin{equation}\label{VNEN tau fixed}
\lim_{ N \to \infty } \frac{V_{N,\tau}} { E_{N,\tau} } = 2-\sqrt{2}.
\end{equation}

As the mean and the variance are of the same order, it is immediate to obtain the convergence
\[
 \frac{ \mathcal{N}_{\tau_N} }{N} \to c(\alpha)
\]
as $N \to \infty$, in probability.

In Theorem \ref{Thm_density}, we derive the limiting empirical distribution of the real eigenvalues. Denote by $\rho_{N,\tau_N}(x)$ the empirical distribution of the real eigenvalues of $X$. Then as $N\to\infty$, 
\begin{equation}
	\rho_{N,\tau_N}(x) \to \rho_\alpha^w(x):=\mathbbm{1}_{ [-2,2] }(x) \cdot \frac{1}{c(\alpha)} \frac{1}{2\alpha \sqrt{\pi}} \erf( \tfrac{\alpha}{2} \sqrt{4-x^2} )
\end{equation}
uniformly on compact subsets of $(-2,2)$. Here $c(\alpha)$ is a normalization constant in \eqref{c alternative}, which turns $\rho_\alpha^w$ into a probability density function. Applying the asymptotic formula \eqref{eq:asymp_erf}, we find that the density $\rho_\alpha^w$ interpolates the semicircle law
\begin{equation} \label{rhosc}
	\rho_{sc}(x):= \mathbbm{1}_{ [-2,2] }(x) \cdot \frac{1}{2\pi} \sqrt{4-x^2}
\end{equation} 
and the uniform distribution $\rho_1^s$ defined in \eqref{rhos}, see Figure~\ref{Fig_GauRealEigenValueHistogram}. 

Our proof for the number of the real eigenvalues is based on a (double) contour integral representation, which follows from the skew-orthogonal polynomial kernel of the associated Pfaffian point process and some basic properties of hypergeometric functions. 

For the density of the real eigenvalues, we use a version of the Christoffel--Darboux identity (Lemma~\ref{Lem_Hermite}) for Hermite polynomials with complex variables. We also estimate certain oscillatory integrals, which naturally arise from the Plancherel--Rotach strong asymptotics. 

In the elliptic case $\tau_N \equiv \tau \in [0,1)$, the classical Mehler's formula (also known as the Poisson kernel) for orthogonal polynomials can be applied to derive the large-$N$ limit of the correlation kernel. In contrast, the leading mathematical challenge in the proof of the almost-Hermitian case is that it should find seemingly non-trivial analytic expressions of the discrete objects under consideration. For this, we can perform the asymptotic analysis, borrowing some ideas from the theory of special functions.

\subsection{Related works}

The Ginibre ensemble was first introduced in \cite{ginibre1965statistical}. The limiting empirical spectral distribution of a Ginibre ensemble is the uniform distribution on the unit disc in the complex plane, known as the circular law. For a general (fixed) $\tau_N \equiv \tau \in [0, 1)$, the circular law is extended to the elliptic law \cite{girko1986elliptic,MR948613}, the uniform distribution on the elliptic disc
\begin{equation} \label{Ellipse}
K:=\{ x+iy: ( \tfrac{x}{1+\tau} )^2+( \tfrac{y}{1-\tau} )^2 \le 1 \},
\end{equation}
see Figure~\ref{Fig_REG}.

The formula \eqref{EN tau<1} was first proved by Edelman, Kostlan, and Shub for a real Ginibre ensemble ($\tau=0$). More precisely, it was shown in \cite[Corollary 5.2]{MR1231689} that 
\begin{equation}\label{E_{N,0} asym}
	E_{N,0}=\Big( \frac{2}{\pi}N \Big)^{\frac12} \Big[ 1-\frac{3}{8N}-\frac{3}{128N^2}+\frac{27}{1024N^3}+\frac{499}{32768N^4}+O(\frac{1}{N^5}) \Big]+\frac12.
\end{equation}
For a fixed value of $\tau \equiv \tau_N \in [0,1)$, Forrester and Nagao \cite{MR2430570} expressed $E_{N,\tau}$ in terms of a summation of hypergeometric functions using the theory of skew-orthogonal polynomials previously developed in \cite{forrester2007eigenvalue} for the case $\tau_N \equiv 0$ from which one can derive \eqref{EN tau<1}. (See also \cite{MR2371225,MR2530159} for the scaling limits of real Ginibre matrices.)

\begin{figure}[!ht]
	\includegraphics[width=\textwidth]{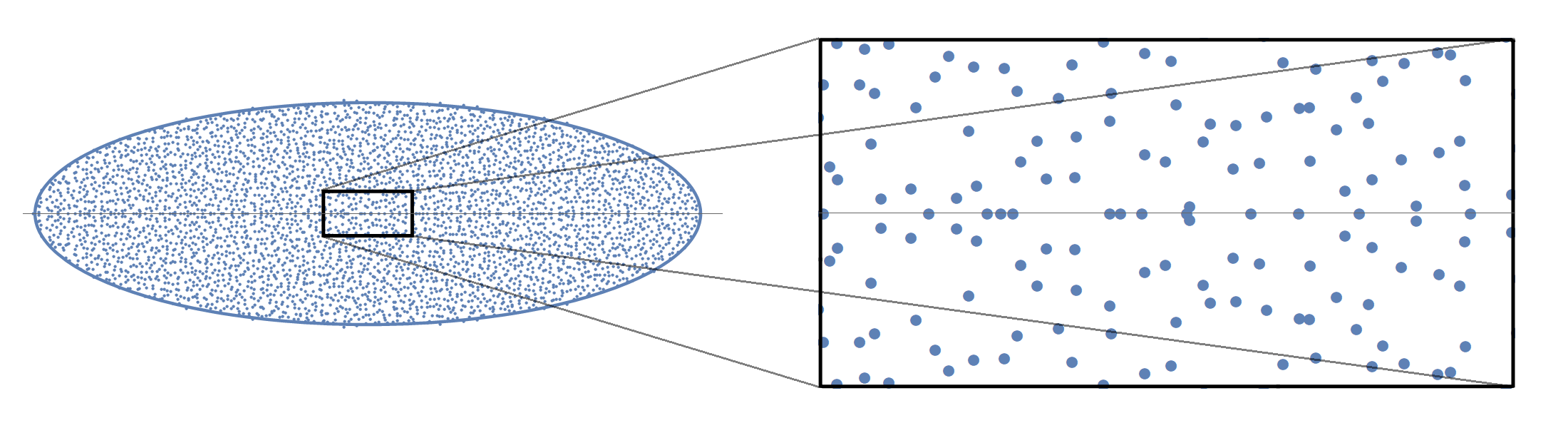}
	\caption{Eigenvalues of $X$, where $\tau_N\equiv 1/2$ and $N=4096$} \label{Fig_REG}
\end{figure}

For the real eigenvalues of other random matrix models, we refer to \cite{MR2185860,MR2430570,MR3685239,little2021number,forrester2020many,MR2788023} and references therein. We also remark that the statistics of real eigenvalues enjoy an intimate relationship with diverse topics, including the Zakharov--Shabat system \cite{MR4068316} and the annihilating Brownian motion \cite{MR3678474}. 

The almost-Hermitian regime was introduced in the series of works \cite{MR1634312,fyodorov1997almost,MR1431718} by Fyodorov, Khoruzhenko, and Sommers. For more details and some physical connotations on such intermediate regime, we refer to some recent works \cite{MR3845296,AB20,FT20} and references therein. The density $\rho_\alpha^w$ was previously found
by Efetov \cite{efetov1997directed} using the supersymmetry method in the context of directed quantum chaos.

In general, the odd $N$ case should be treated separately, see e.g., \cite{MR2485724,MR2341601}. Nevertheless, our assumption that $N$ be an even integer is merely for convenience to apply main results in \cite{MR2430570}. It is to be expected that such an assumption should not be necessary for our main results.

Our model can be generalized to real elliptic matrices with non-Gaussian entries. Even with non-Gaussian entries, the circular law and the elliptic law hold both at the global and local levels. It is known as the universality of the circular law \cite{MR2663633,MR2722794,MR3306005,MR3230002}, and the elliptic law \cite{MR3403996,alt2021local}, respectively. The universality of the leading order asymptotics of the expected number of real eigenvalues also holds for a class of i.i.d. random matrices (i.e., $\tau_N \equiv 0$); see \cite{MR3306005}. It is expected that a similar result also holds with general $\tau_N$ including the almost-Hermitian regime; see Figure~\ref{Fig_EN-over-N-alpha-one}. We will discuss these in a future paper.

	\begin{figure}[!ht]
			\begin{subfigure}[h]{0.49\textwidth}
			\includegraphics[width=\textwidth]{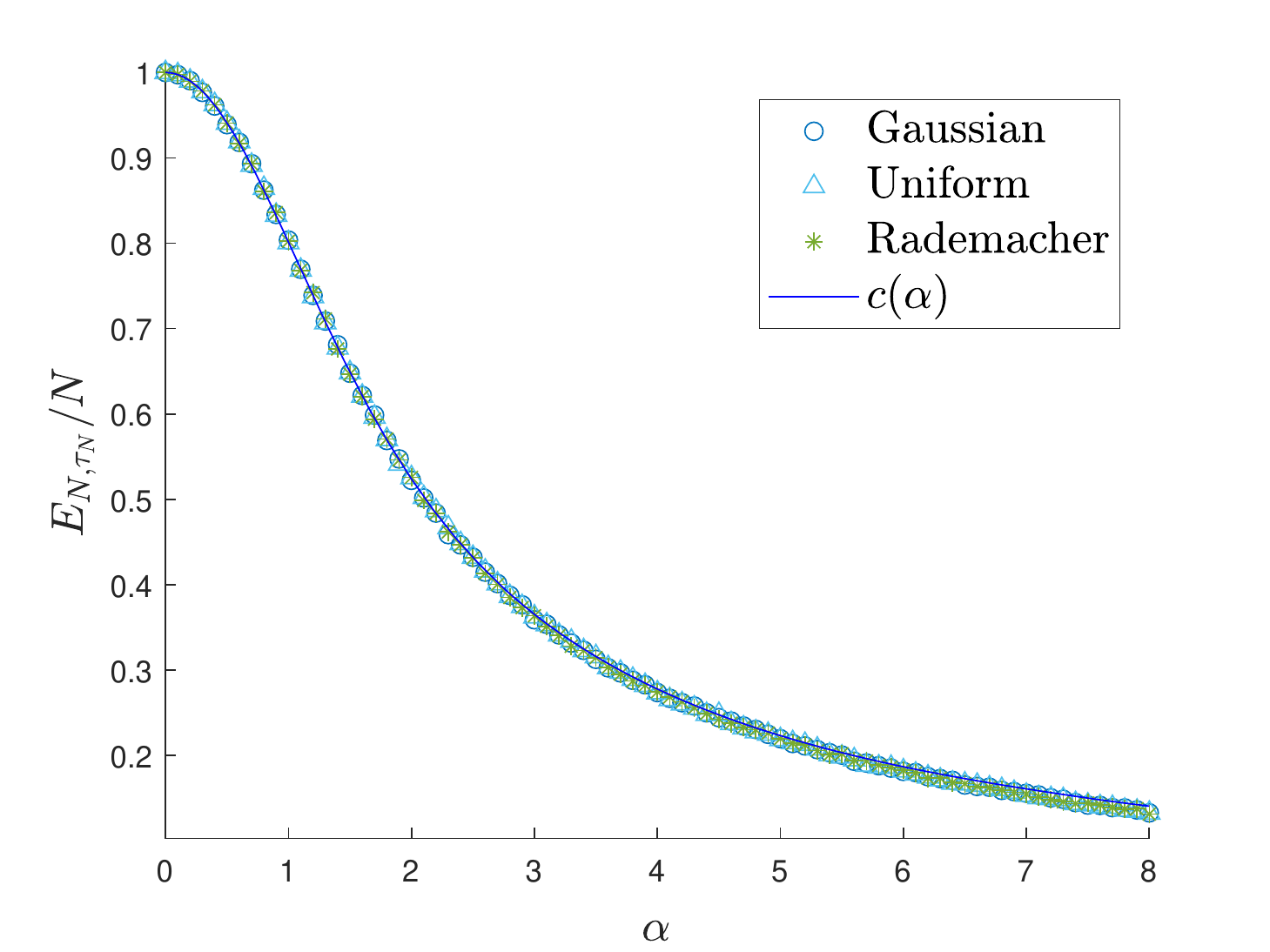} 
			\caption{$c(\alpha)$} 
		\end{subfigure} 
	\begin{subfigure}[h]{0.49\textwidth}
	\includegraphics[width=\textwidth]{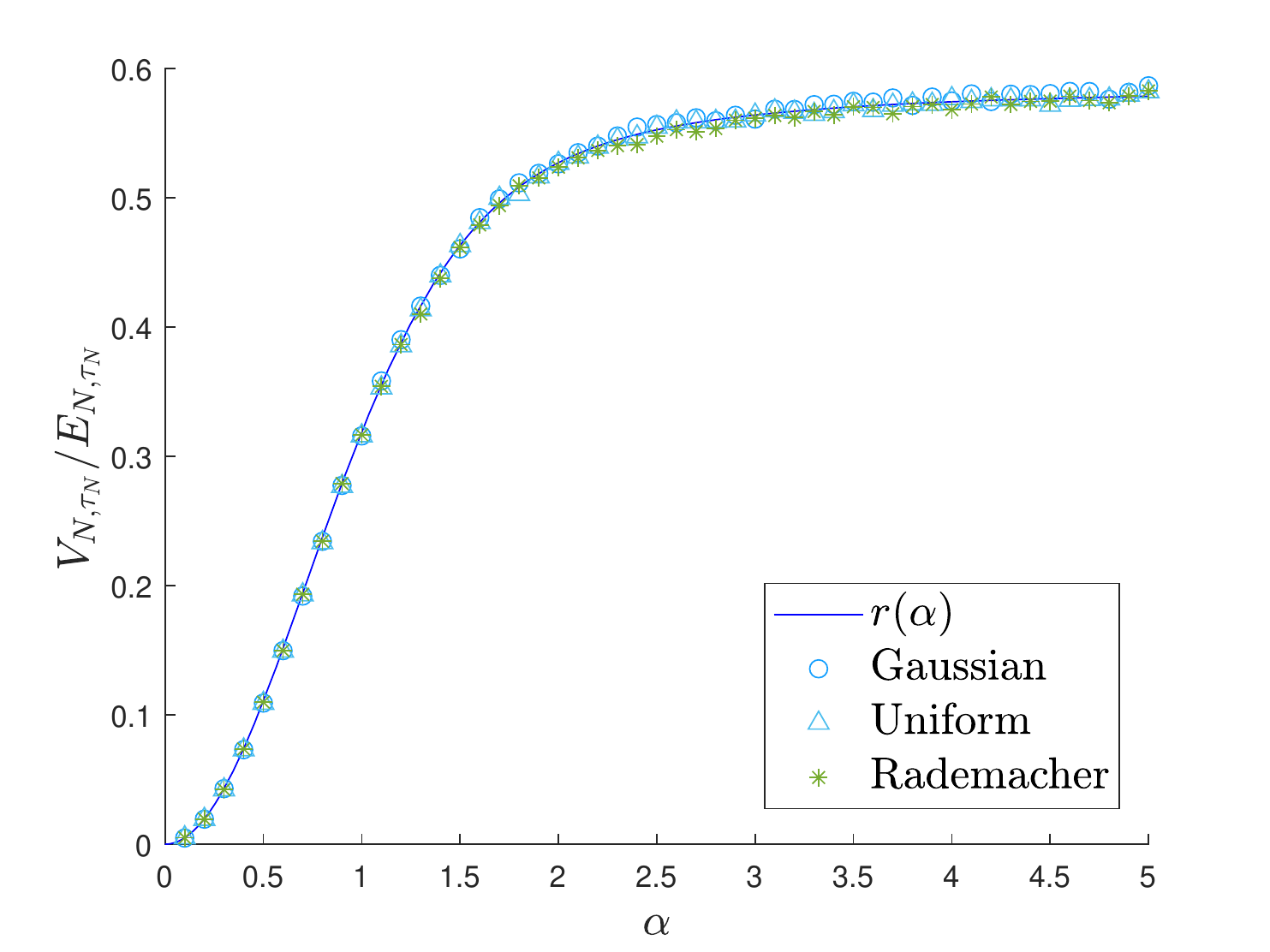}
		\caption{$r(\alpha)$} 
	\end{subfigure}	
		\caption{(A): The plot displays numbers of real eigenvalues divided by $N$ for $256$ samples of random matrices of size $N=256$ with various entries and their comparison with $c(\alpha)$. Here the entries of random matrices are given by Gaussian, uniform, and Rademacher distributions. (B): The graph indicates the variance of real eigenvalues divided by its number, for $10^5$ samples of random matrices of size $N=64$ with entries given by the identical distributions in (A) and their comparison with $r(\alpha)$. }
		\label{Fig_EN-over-N-alpha-one} \label{Fig_VNoverEN}
	\end{figure}

For the non-Hermitian Wishart ensembles (also known as chiral Ginibre ensembles or sample cross-covariance matrices), it was shown in \cite{MR2679835} that the associated skew-orthogonal polynomials can be expressed in terms of the generalized Laguerre polynomials. Using this, the correlation kernel of the associated Pfaffian point process was studied in \cite{MR2592354}. 
Based on this a priori knowledge, one may consider similar problems in this work for the non-Hermitian Wishart ensembles. 
In particular, one can expect that the limiting empirical distributions in the almost-Hermitian regime interpolate the Marchenko-Pastur law with the ``shifted elliptic law'' recently discovered in \cite{MR4229527}. 
We will address these problems in future work.

\subsection{Organization of the paper}

The rest of the paper is organized as follows.
In Section~\ref{Section_Main Results}, we introduce the precise definition of the model and state the main results. The main results on the real eigenvalues of the model in the almost-Hermitian regime are Theorems~\ref{Thm_EN}, ~\ref{Thm_Variance} and ~\ref{Thm_density}. 
We also present the counterpart of Theorem~\ref{Thm_EN} for the elliptic regime in Proposition~\ref{Thm_EN t fixed}. In Section~\ref{Section_expectation}, we discuss a useful integral representation of the expected number of the real eigenvalues and prove Theorem~\ref{Thm_EN} and Proposition~\ref{Thm_EN t fixed}. Section~\ref{Section_distributions} is devoted to carrying out some asymptotic analysis on the pre-kernel and to proving Theorem~\ref{Thm_density}. 
In Section~\ref{Section_Variance}, we prove Theorem~\ref{Thm_Variance}, based on a similar strategy as in Section~\ref{Section_expectation}, using some results from Section~\ref{Section_distributions} as well.

\section{Main Results}\label{Section_Main Results}

\subsection{Elliptic matrix}

In this subsection, we define the elliptic matrix, the primary model we consider in this work.
We consider a sequence $\tau_N \in [0,1]$ called a \emph{non-Hermiticity} parameter. In the sequel we write $\tau=\lim_{ N \to \infty } \tau_N$, which is always assumed to be well defined. 
Let $\{\xi_{j,N}\}_{N=1}^\infty$ ($j=0,1,2$) be sequences of real random variables, which satisfy 
\begin{equation} \label{xi j}
\mathbb{E} \xi_{j,N}=0, \qquad \mathbb{E} \xi_{j,N}^2=1, \qquad \mathbb{E} \xi_{j,N}^k< \infty, \qquad \mathbb{E} \xi_{1,N} \xi_{2,N} = \tau_N,
\end{equation}
where $k \in \mathbb{N}$. 
We consider a random matrix $X=(x_{jk})_{j,k=1}^N$ whose entries are given as follows:
\begin{itemize}
	\item $\{ (x_{jk},x_{kj}) \}_{j<k} \cup \{x_{jj}\}_{j} $ consists of independent random variables;
	\item $\{x_{jj}\}_{j}$ are independent copies of $\frac{1}{\sqrt{N}}\xi_{0,N}$;
	\item $\{ (x_{jk},x_{kj}) \}_{j<k}$ are independent copies of $\frac{1}{\sqrt{N}}(\xi_{1,N},\xi_{2,N})$.
\end{itemize}
We assume that $\xi_{j,N}$'s are Gaussian, which coincides with $X$ in \eqref{X GOE antiGOE}.

\subsection{Main Theorems}

Our first main result, Theorem~\ref{Thm_EN}, is on the asymptotic expansion of $E_{N,\tau_N}$ in the almost-Hermitian regime. Recall that the modified Bessel function $I_\nu$ of the first kind is given by 
\begin{equation}\label{I nu}
	I_\nu(z):=\sum_{k=0}^\infty \frac{(z/2)^{2k+\nu}}{ k! \, \Gamma(\nu+k+1) }.
\end{equation}

\begin{thm} \label{Thm_EN}
	Let $\tau_N=1-\tfrac{\alpha^2}{N}$. Then for any positive integer $m \ge 2$,
	\begin{equation}\label{EN expansion}
		E_{N,\tau_N}=N c(\alpha)+c_0(\alpha)+\frac12+\sum_{l=1}^{m-1} \frac{c_l(\alpha)}{N^l}+O( \frac{1}{N^m} )
	\end{equation}
as $N \to \infty$, where 
\begin{equation}
	\label{c(alpha)}
c(\alpha):= e^{-\alpha^2/2} [ I_0( \tfrac{\alpha^2}{2} )+ I_1( \tfrac{\alpha^2}{2} ) ],
\end{equation}
\begin{equation}
c_0(\alpha):=-\tfrac12 e^{-\alpha^2/2} [ I_0( \tfrac{\alpha^2}{2} )+\alpha^2 I_1( \tfrac{\alpha^2}{2} ) ],
\end{equation}
and for $l \ge 1$, 
\begin{equation}
	c_l(\alpha):=\sum_{k=0}^{l+1} \sum_{m=l+1-k}^\infty	q_{m,l+1-k} \, \frac{(-1)^{m-l+k}}{(m-l+k)!} \frac{(2k-3)!!}{(2k)!!} \frac{(2m-1)!! }{ m! } \Big( \frac{\alpha^2}{2} \Big)^{k+m}.
\end{equation}
Here $q_{k,s}$ is defined by the generating function
\begin{equation} \label{q_k,s}
	\Big( \frac{t\, e^{t}}{ e^t-1 } \Big)^{k+2} \frac{2}{e^t+1} =\sum_{s=0}^\infty q_{k,s} \, t^s.
\end{equation}

\end{thm}

We remark that the definition of $c(\alpha)$ in \eqref{c(alpha)} coincides with that in \eqref{c alternative}; see Remark \ref{rem:c_alpha}.

For a fixed $\tau_N \equiv\tau \in [0,1)$, we have the following full expansion of $E_{N,\tau}$. 

\begin{prop}\label{Thm_EN t fixed}
	Let $\tau_N \equiv \tau \in [0,1)$ be fixed. 
	Then for any positive integer $m$, we have 
	\begin{equation} \label{EN tau}
	 	E_{N,\tau}=\Big( \frac{2}{\pi} \frac{1+\tau}{1-\tau} N \Big)^{\frac12}\Big[ 1+\sum_{l=1}^{m-1} \frac{ a_l(\tau) }{N^{l}}+O(\frac{1}{N^m}) \Big]+\frac12, 
	\end{equation}
	as $N\to\infty$, where 
	\begin{equation}
	 a_l(\tau):=-\sqrt{1-\tau} \, \frac{(2l-3)!!}{2^l} \sum_{k=0}^\infty \frac{(2k-1)!!}{2^k \,k!} \wh{p}_{k,l} \, \tau^k.
	\end{equation}
	Here $\wh{p}_{k,l}$ is defined by the generating function
	\begin{equation} \label{p wh ks}
	\Big( \frac{e^t-1}{t} \Big)^{-\frac32} \frac{ 2\, e^{t(k+2)} }{ e^t+1 } =\sum_{l=0}^\infty \wh{p}_{k,l} \, t^l.
	\end{equation}
\end{prop}

Our next object of interest is the variance $V_{N,\tau_N}$ of the number $\mathcal{N}_{\tau_N}$ of real eigenvalues in the almost-Hermitian regime.

\begin{thm}\label{Thm_Variance}
	Let $\tau_N=1-\tfrac{\alpha^2}{N}$. Then we have 
	\begin{equation} \label{VNEN}
	\lim_{ N \to \infty }	\frac{V_{N,\tau_N}}{E_{N,\tau_N} } =r(\alpha):= 2-2 \, e^{-\alpha^2/2} \frac{ I_0( \alpha^2 )+ I_1( \alpha^2 ) }{ I_0( \tfrac{\alpha^2}{2} )+ I_1( \tfrac{\alpha^2}{2} ) }.
	\end{equation}
\end{thm}

Note that the behavior \eqref{r limits} of $r(\alpha)$ can be easily checked using the well-known asymptotic 
\[
I_\nu(x) \sim 
\begin{cases}
	\frac{(x/2)^\nu}{\Gamma(\nu+1)} &\text{as} \quad x \to 0,
	\\
	\frac{e^x}{\sqrt{2\pi x}} & \text{as} \quad x \to \infty,
\end{cases}
\]
see e.g., \cite[Eq.(10.30.1), (10.30.4)]{olver2010nist}.

As an immediate consequence of Theorems~\ref{Thm_EN} and~\ref{Thm_Variance}, we obtain the convergence of the random variables $\mathcal{N}_{\tau_N}/N$ in probability as follows. 

\begin{cor}
	Let $\tau_N=1-\tfrac{\alpha^2}{N}$. Then we have
\begin{equation}
	\frac{\mathcal{N}_{\tau_N}}{N} \to c(\alpha) 
\end{equation}
as $N\to \infty$, in probability. 
\end{cor}

\begin{rmk}
 The functions $a_l$ in Proposition~\ref{Thm_EN t fixed} for the first few values of $l$ are given as follows:
 \begin{align*}
	&a_1(\tau):=\tfrac{\tau-3}{8(1-\tau)}, \hspace{7.5em} a_2(\tau):=\tfrac{5\tau^2-14\tau-3}{128(1-\tau)^2}, 
	\\
	&a_3(\tau):=\tfrac{-17\tau^3+73\tau^2-203\tau+27}{1024(1-\tau)^3},
	\qquad a_4(\tau):=\tfrac{-541\tau^4+2684\tau^3-6846\tau^2-4196\tau+499}{32768(1-\tau)^4}. 
\end{align*}
Note that the asymptotic expansion \eqref{EN tau} for $\tau=0$ recovers \eqref{E_{N,0} asym}.
	
	We also remark that the functions $a_l$ can be written as $a_l(\tau)=P_l( \tfrac{\tau}{\tau-1} ),$
	where $P_l$ is a polynomial of degree $l$ given by
	\begin{equation}
		P_l(x):=\frac{(2l-3)!!}{2^{l}}\sum_{k=0}^{l} (-1)^{k+1}\frac{(2k-1)!!}{2^{k} k! } p_{k,l-k} \,x^k.
	\end{equation}
	Here $p_{k,s}$ is defined by the generating function
	\begin{equation}
		\Big( \frac{e^t-1}{t} \Big)^{k-\frac32} \frac{2\,e^{2t}}{e^t+1} =\sum_{s=0}^\infty p_{k,s} \, t^s.
	\end{equation} 
\end{rmk}

\begin{rmk} In general, the functions $c_l$'s are of the form
\begin{equation}
	c_l(\alpha)= e^{-\alpha^2/2} [ P_{l,0}(\alpha) \, I_0( \tfrac{\alpha^2}{2} )+P_{l,1}(\alpha)\, I_1( \tfrac{\alpha^2}{2} ) ],
\end{equation}
where $P_{l,1}, P_{l,2}$ are some even polynomials of degree $4l+2$. For example, we have 
$$
P_{1,0}(\alpha)=-\tfrac{\alpha^4(3\alpha^2-8)}{48}, \qquad P_{1,1}(\alpha)=\tfrac{\alpha^2(3\alpha^4-8\alpha^2-2)}{48}
$$
and 
\begin{align*}
P_{2,0}(\alpha)=\tfrac{\alpha^4(\alpha^{6}-8\alpha^4+11\alpha^2+1)}{96}, \qquad P_{2,1}(\alpha)=-\tfrac{\alpha^2(\alpha^{8}-7\alpha^6+6\alpha^4+3\alpha^2+4)}{96}.
\end{align*}
\end{rmk}

\begin{rmk} \label{rem:c_alpha}
To see the equivalence of the definitions of $c(\alpha)$, recall that the error function has the power series expansion (see \cite[Eq.(7.6.1)]{olver2010nist})
	\[
		\erf(z)=\frac{2}{\sqrt{\pi}} \sum_{n=0}^{\infty}\frac{(-1)^n }{ n! \, (2n+1) } z^{2n+1}.
	\]
	Then the expression \eqref{c(alpha)} can be obtained from \eqref{c alternative} by straightforward computations using
	\[
	\int_0^1 (1-s^2)^{n+\frac12} \,ds= \frac{\sqrt{\pi}}{2} \frac{\Gamma(n+\frac32)}{(n+1)!}.
	\]
\end{rmk}

We now discuss the density $\rho_N \equiv \rho_{N,\tau_N}$ of real eigenvalues. 
In the almost-Hermitian regime, we obtain the following theorem, which features a one-parameter family of probability distributions $\rho_\alpha^w$ interpolating between the Wigner semicircle law in \eqref{rhos} and the uniform distribution in \eqref{rhosc} (i.e., the elliptic law restricted on the real axis).

\begin{thm}\label{Thm_density}
	Let $\tau_N=1-\tfrac{\alpha^2}{N}$. Then as $N\to\infty$, 
	\begin{equation}
		\label{Goal2}
		\rho_{N,\tau_N}(x) \to \rho_\alpha^w(x):=\mathbbm{1}_{ [-2,2] }(x) \cdot \frac{1}{c(\alpha)} \frac{1}{2\alpha \sqrt{\pi}} \erf( \tfrac{\alpha}{2} \sqrt{4-x^2} )
	\end{equation}
	uniformly on compact subsets of $(-2,2).$ Here $c(\alpha)$ is a normalization constant given by \eqref{c alternative}, which turns $\rho_\alpha^w$ into a probability density function.
\end{thm}

\begin{figure}[h]
	\begin{center}
		\begin{subfigure}[h]{0.24\textwidth}
			\includegraphics[width=\textwidth]{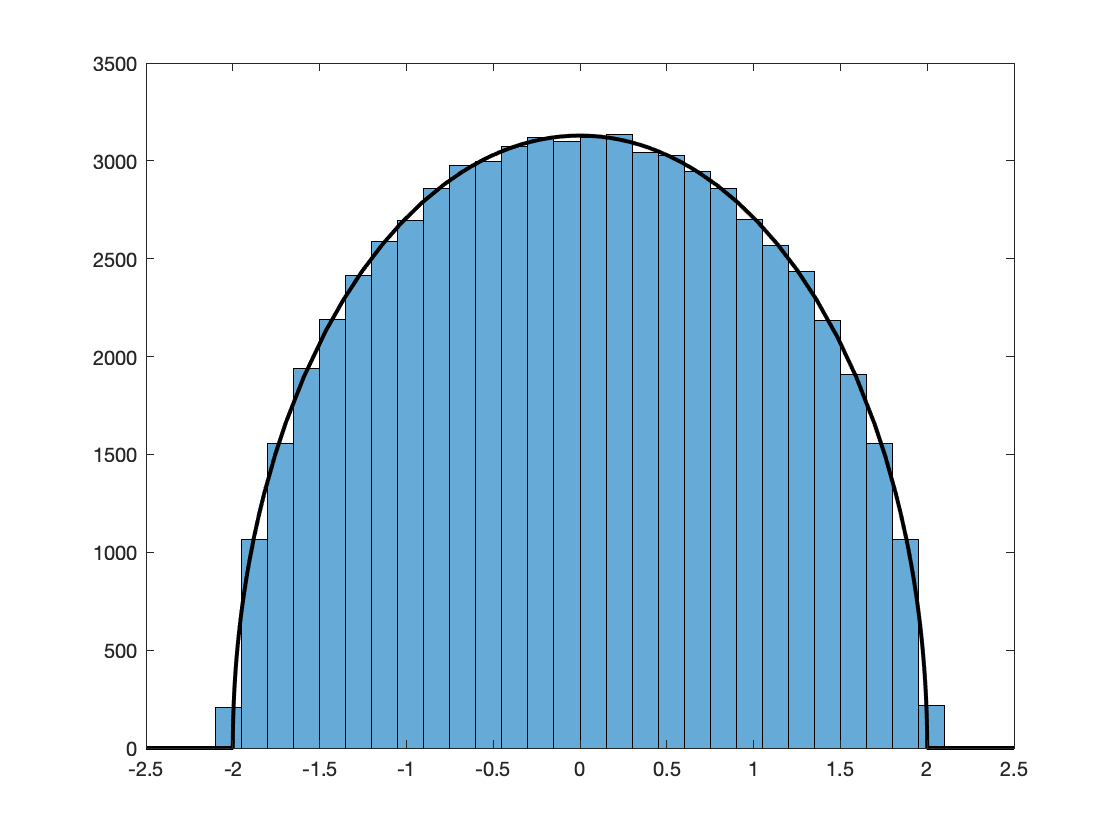}
			\caption{$\alpha=0$}
		\end{subfigure} 
		\begin{subfigure}[h]{0.24\textwidth}
			\includegraphics[width=\textwidth]{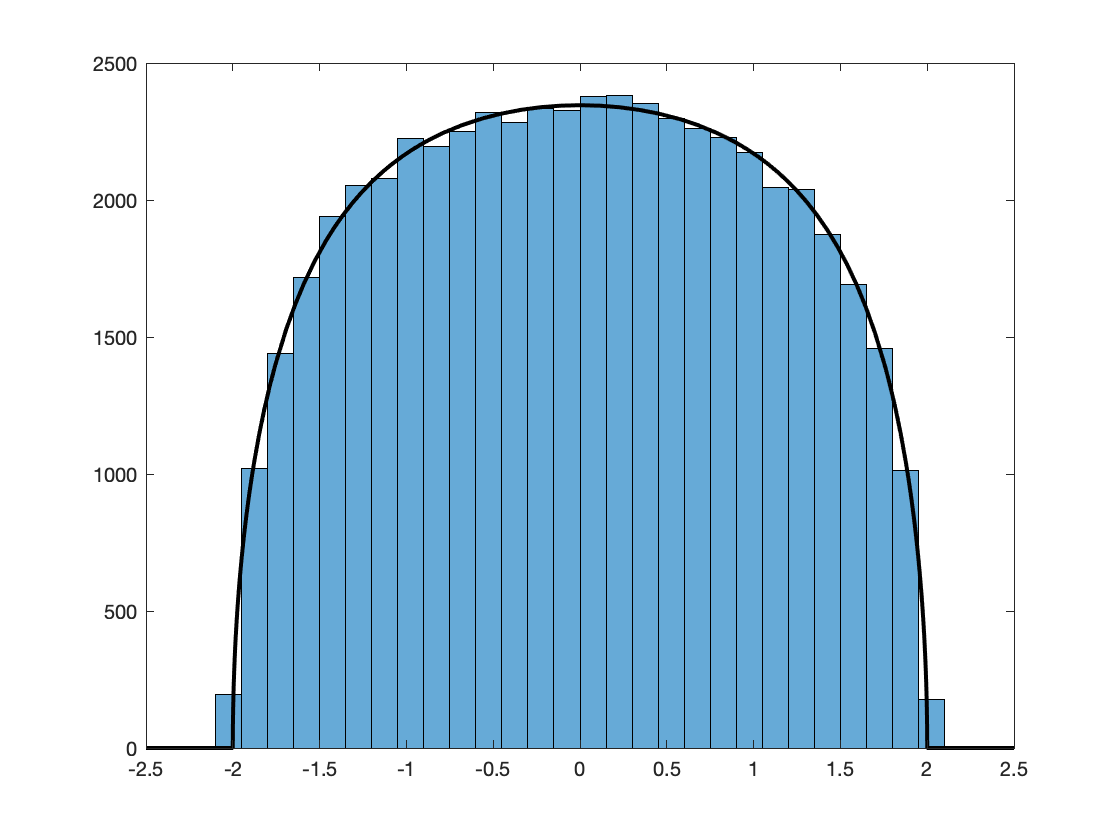}
			\caption{$\alpha=1$}
		\end{subfigure} 
		\begin{subfigure}[h]{0.24\textwidth}
			\includegraphics[width=\textwidth]{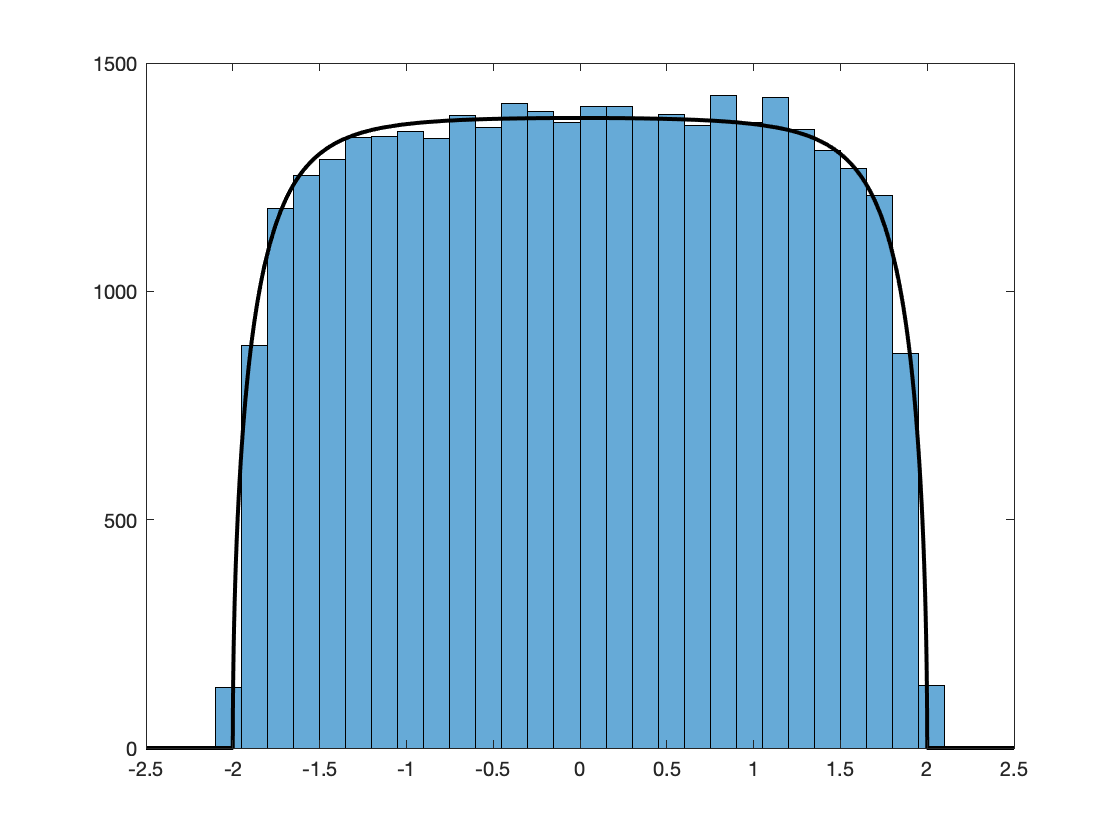}
			\caption{$\alpha=2$}
		\end{subfigure} 
		\begin{subfigure}[h]{0.24\textwidth}
			\includegraphics[width=\textwidth]{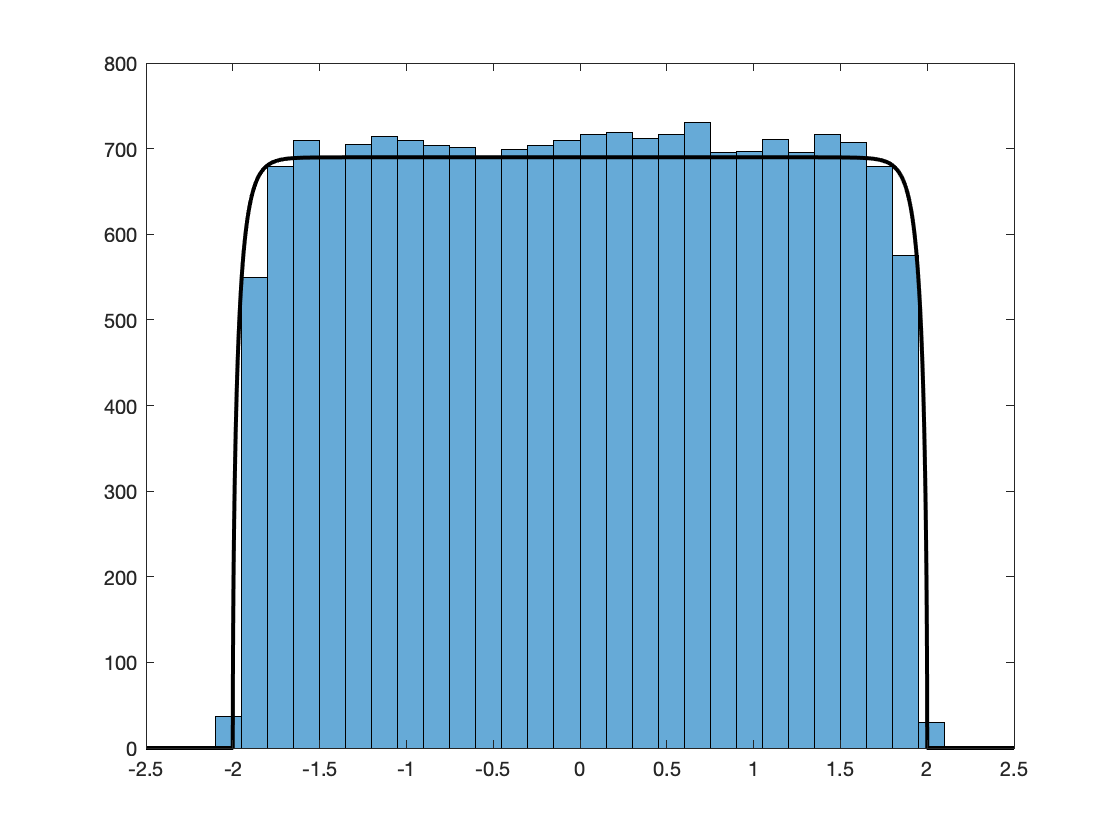}
			\caption{$\alpha=4$}
		\end{subfigure} 
	\end{center}
	\caption{ The plots display histograms of real eigenvalues of $X$. The full line in each histogram is the graph of $\rho_\alpha^w$ with an appropriate scale. All the histograms are drawn from $256$ samples with $N=256$.}
	\label{Fig_GauRealEigenValueHistogram}
\end{figure}

\section{Expected number of real eigenvalues}\label{Section_expectation}

In this section, we prove Theorem~\ref{Thm_EN} and Proposition~\ref{Thm_EN t fixed}. Our proof is based on the following representation of $E_{N,\tau_N}$ shown in \cite{MR2430570}: for any $\tau_N < 1$,
\begin{align} \label{EN}
	\begin{split}
		E_{N,\tau_N}&=\Big(\frac{2}{\pi} \frac{1+\tau_N}{1-\tau_N} \Big)^{\frac12} \sum_{k=0}^{N/2-1} \frac{\Gamma(2k+\frac12)}{(2k)!} {}_2 F_1( \tfrac12, \tfrac12; -2k+\tfrac12; -\tfrac{\tau_N}{1-\tau_N} ),
	\end{split}
\end{align}
where ${}_2 F_1$ is the hypergeometric function defined by the Gauss series 
\begin{equation} \label{2F1 Gauss series}
	{}_2 F_1(a,b;c;z):=\frac{\Gamma(c)}{\Gamma(a)\Gamma(b)} \sum_{s=0}^\infty \frac{\Gamma(a+s) \Gamma(b+s) }{ \Gamma(c+s) s! } z^s, \quad (|z|<1)
\end{equation}
and by analytic continuation elsewhere. 

Our strategy for the proof of Theorem~\ref{Thm_EN} is summarized as follows:
\begin{itemize}
 \item we obtain a contour integral representation of $E_{N,\tau_N}$, which holds for any $\tau_N<1$ (Lemma~\ref{Prop_gN}); 
 \item we treat the integrand in the previous step as a function of variables $N$ and $\alpha$, and compute its power series expansion;
 \item we derive the large-$N$ expansion of $E_{N,\tau_N}$ by the residue calculus and express its coefficients in terms of the modified Bessel functions.
\end{itemize}

To analyze the right-hand side of the identity \eqref{EN}, we begin with the following lemma, which provides a contour integral representation of $E_{N,\tau_N}$.

\begin{lem}\label{Prop_gN}
For any $N>1$ and $\tau_N \in [0,1)$, we have 
\begin{align}
\begin{split}
\label{EN int rep}
	E_{N,\tau_N}&=\Big(\frac{2}{\pi} \frac{1+\tau_N}{1-\tau_N} \Big)^{\frac12} g_N( -\tfrac{\tau_N}{1-\tau_N} ),
\end{split}
\end{align}
where $g_N: \R_{-} \to \R_{+}$ is defined by
\begin{equation} \label{g_N}
	g_N(x):=\sqrt{\pi } \,\, \underset{\zeta=1}{\textup{Res}} \, \Big[ \frac{\zeta^{-\frac32} (2-\zeta)^{-1} (1-\zeta x )^{-\frac12} }{ (\zeta-1)^{N-1} } \Big].
\end{equation}
\end{lem}
\begin{proof}
We first analyze the summand in \eqref{EN}. 
For $|\arg(1-z)|<\pi$ and $b,1-c,c-b \not\in \mathbb{N},$ the hypergeometric function ${}_2F_1(a,b;c;z)$ has an integral representation
\begin{align}\label{2F1 int rep}
	\begin{split}
	{}_2 F_1(a,b;c;z)&=-\frac{\Gamma(c)}{\Gamma(b) \Gamma(c-b) } \frac{ e^{ -i\pi c} }{ 4 \sin (\pi b) \sin (\pi (c-b) ) }
	\\
	&\quad \times \oint_{ \mathcal{P}(0,1) } \zeta^{b-1} (1-\zeta)^{c-b-1} (1-\zeta z)^{-a}\,d\zeta, 
	\end{split}
\end{align}
see e.g., \cite[Eq.(15.6.5)]{olver2010nist}. Here $\mathcal{P}(0,1)$ is a Pochhammer contour entwining $0$ and $1$. 
It follows from the reflection formula of Gamma function
\begin{equation}\label{Gamma reflection}
\Gamma(z)\Gamma(1-z)=\pi/\sin(\pi z)
\end{equation} 
that for $\kappa \in \C \setminus \mathbb{Z}$, 
\begin{equation*}
\frac{\Gamma(2\kappa+\frac12)}{(2\kappa)!} \frac{ \Gamma(-2\kappa+\frac12) }{ \Gamma(\frac12) \Gamma(-2\kappa) } \frac{ e^{ -i\pi/2 } }{ 4 \sin (2\kappa \pi) }=\frac{i}{4\sqrt{\pi}} \sec(2\kappa \pi), 
\end{equation*}
which has a removable singularity at each $\kappa \in \mathbb{Z}$. Combining the equations above, for $x<0$ we have 
\begin{align*}
\frac{\Gamma(2k+\frac12)}{(2k)!} {}_2 F_1( \tfrac12, \tfrac12; -2k+\tfrac12; x )&=\frac{i}{4\sqrt{\pi}} \oint_{ \mathcal{P}(0,1) } \zeta^{-\frac12} (1-\zeta)^{-2k-1} ( 1-\zeta x )^{-\frac12}\,d\zeta\\
&=-\sqrt{\pi} \, \, \underset{\zeta=1}{\text{Res}} \, \Big[ \zeta^{-\frac12} (1-\zeta)^{-2k-1} (1-\zeta x)^{-\frac12}\Big].
\end{align*}
Here the second identity follows from $-2k-1 \in \mathbb{Z}$ and the standard deformation of the Pochhammer contour. 

Taking the sum of the above, we observe that only the second term of the right-hand side below 
$$
\sum_{k=0}^{N/2-1} (1-\zeta)^{-2k-1}=\frac{1-\zeta}{\zeta(\zeta-2)}-\frac{1}{(1-\zeta)^{N-1}\zeta(\zeta-2)} ,
$$
contributes to the residue.
Thus we have 
\begin{align}
	\begin{split}
	\sum_{k=0}^{N/2-1} \frac{\Gamma(2k+\frac12)}{(2k)!} {}_2 F_1( \tfrac12, \tfrac12; -2k+\tfrac12; x )=g_N(x).
	\end{split}
\end{align}
Combining this with the expression \eqref{EN}, the proof is complete.
\end{proof}

\begin{rmk}
When $\tau_N \equiv 0$, it follows from ${}_2F_1(a,b;c;0)=1$ and the duplication formula of Gamma function that the expression \eqref{EN} is simplified to 
\begin{align} \label{EN tau=0}
	\begin{split}
		E_{N,0}=\sqrt{2} \sum_{k=0}^{N/2-1} \frac{(4k-1)!!}{(4k)!!}
		&=\frac12+\sqrt{\frac{2}{\pi}} \frac{\Gamma(N+\frac12)}{(N-1)!} {}_2 F_1( 1,-\tfrac12;N;\tfrac12 ).
	\end{split}
\end{align}
The formula \eqref{EN tau=0} also appears in \cite[Corollary 5.1, 5.3]{MR1231689}.
For $\tau_N\equiv 0$, one can recognize the right-hand side of \eqref{EN int rep} in terms of a hypergeometric function using \eqref{2F1 int rep}, namely, we have
\begin{equation}
	g_N(0)= \frac{ \Gamma(N-\frac12) }{ (N-2)! } 	{}_2 F_1(1,-\tfrac12;-N+\tfrac32;\tfrac12). 
\end{equation}
Note that the regularized hypergeometric function ${}_2\textup{\textbf{F}}_1(a,b;c;z):=\frac{1}{\Gamma(c)} {}_2F_1(a,b;c;z)$ satisfies the linear transform 
\begin{align} \label{2F1 linear trans}
\begin{split}
\frac{ \sin(\pi(c-b)) }{\pi} {}_2\textup{\textbf{F}}_1(1,b;c;z)&=\frac{ 1 }{ \Gamma(b) \Gamma(c-b) } (1-z)^{c-1-b} z^{1-c}
\\
&- \frac{1}{\Gamma(c-1)\Gamma(c-b)}{}_2\textup{\textbf{F}}_1(1,b;b-c+2;1-z),
\end{split}
\end{align}
see \cite[Eq.(15.4.6), (15.8.6)]{olver2010nist}. Using this, one can observe that Lemma~\ref{Prop_gN} for $\tau_N \equiv 0$ is equivalent to \eqref{EN tau=0}. 
\end{rmk}

It is instructive first to present the proof of Proposition~\ref{Thm_EN t fixed}, which includes essential ideas for Theorem~\ref{Thm_EN} but requires fewer computations.

\begin{proof}[Proof of Proposition~\ref{Thm_EN t fixed}]

 Using the elementary binomial expansion
	$$
		\Big( \frac{1-\zeta x}{1-x} \Big)^{-\frac12} = \sum_{k=0}^\infty \frac{(2k-1)!!}{2^k k!} \Big((1-\zeta) \frac{x}{x-1} \Big)^k, \qquad \Big(\Big| (1-\zeta) \frac{x}{x-1} \Big|<1 \Big),
	$$
	we write $g_N$ in \eqref{g_N} as
	\begin{equation}
	 g_N(x)=\sqrt{ \frac{\pi}{1-x} } \,\, \sum_{k=0}^\infty (-1)^k \frac{(2k-1)!!}{2^k k!} \underset{\zeta=1}{\textup{Res}} \, \Big[ \frac{\zeta^{-\frac32} (2-\zeta)^{-1} }{ (\zeta-1)^{N-k-1} } \Big] \Big( \frac{x}{x-1} \Big)^k. 
	\end{equation}
Moreover, by \eqref{2F1 int rep} and \eqref{2F1 linear trans}, the residue in the right-hand side of this identity is computed as 
	\begin{align*}
 \underset{\zeta=1}{\textup{Res}} \, \Big[ \frac{\zeta^{-\frac32} (2-\zeta)^{-1} }{ (\zeta-1)^{N-k-1} } \Big] &=\frac{1}{2\sqrt{2}}+\frac{ \sqrt{\pi} }{ \Gamma(\frac12-N+k) } {}_2\textup{\textbf{F}}_1(1,-\tfrac12;N-k;\tfrac12).
	\end{align*}
Combining above equations with the identity
 $$
 \sum_{k=0}^{\infty} (-1)^k \frac{(2k-1)!!}{k!}\,\Big( \frac{x}{x-1} \Big)^k = \sqrt{\frac{1-x}{1-2x}}, \qquad (x<0),
 $$
	we obtain that for $x<0$, 
 \begin{align*}
		g_N(x)
		&=\frac{\sqrt{\pi}}{2\sqrt{2}} \frac{1}{\sqrt{1-2x}}
		\\
		&+\frac{\pi}{ \sqrt{1-x} }\, \sum_{k=0}^\infty \frac{(2k-1)!!}{2^k k!} \frac{ (-1)^k }{ \Gamma(\frac12-N+k) } {}_2\textup{\textbf{F}}_1(1,-\tfrac12;N-k;\tfrac12) \Big( \frac{x}{x-1} \Big)^k.
	\end{align*}
 
	Now let us recall that as $\lambda \to \infty$ ($\lambda \in \R$)
	\begin{equation}\label{2F1 asymp}
		{}_2 \textup{\textbf{F}}_1(a,b;c+\lambda;z) = \frac{1}{ \Gamma(c-b+\lambda) } \sum_{s=0}^{m-1} q_s(z) \frac{\Gamma(b+s)}{\Gamma(b)} \,\lambda^{-s-b}+O(\lambda^{-m-b}), 
	\end{equation}
	where $m$ is any positive integer, 	see \cite[Eq.(15.12.3)]{olver2010nist}.	 Here $q_0(z)=1$ and $q_s(z)$ when $s=1,2,\ldots$ are defined by the generating function
	\begin{equation}
		\Big( \frac{e^t-1}{t} \Big)^{b-1} e^{t (1-c) } (1-z+ze^{-t})^{-a} =\sum_{s=0}^\infty q_s(z) \, t^s.
	\end{equation}
Using \eqref{2F1 asymp}, we have 
	$$
	{}_2\textup{\textbf{F}}_1(1,-\tfrac12;N-k;\tfrac12)= \frac{\sqrt{N}}{\Gamma(N-k+\frac12)} \Big(\sum_{s=0}^{m-1} \wh{p}_{k,s} \frac{\Gamma(-\frac12+s)}{\Gamma(-\frac12)} N^{-s}+O( N^{-m} ) \Big),
	$$
	where $\wh{p}_{k,s}$ is defined by \eqref{p wh ks}. 
	Then by \eqref{Gamma reflection}, we obtain 
	\begin{align*}
		\frac{ (-1)^k }{ \Gamma(\frac12-N+k) } {}_2\textup{\textbf{F}}_1(1,-\tfrac12;N-k;\tfrac12) = -\frac{\sqrt{N}}{\pi}\,\Big(\sum_{s=0}^{m-1} \wh{p}_{k,s} \,\frac{ (2s-3)!! }{ 2^s N^{s}} +O(\frac{1}{N^{m}})\Big),
	\end{align*}
		which leads to
	 \begin{align*}
		g_N(x)
		&=\frac{\sqrt{\pi}}{2\sqrt{2}} \frac{1}{\sqrt{1-2x}}
		\\
		&- \sqrt{ \frac{N}{1-x} } \, \sum_{k=0}^\infty \frac{(2k-1)!!}{2^k k!} \wh{p}_{k,s} \Big( \frac{x}{x-1} \Big)^k \,\Big(\sum_{s=0}^{m-1} \frac{ (2s-3)!! }{ 2^s N^{s}} +O(\frac{1}{N^{m}})\Big) .
	\end{align*}
	Now Proposition~\ref{Thm_EN t fixed} follows from Lemma~\ref{Prop_gN}. 
\end{proof}

Now we begin to prove Theorem~\ref{Thm_EN}. For each $k=0,1,\ldots,$ and $N>1$, let
\begin{equation} \label{a_{N,k}}
	a_{N,k}(\zeta):= \frac{2}{N^{k+1}} \, \frac{ \zeta^{-k-2} (2-\zeta)^{-1} }{ (\zeta-1)^{N-k-1} } . 
\end{equation}
We will need the following evaluation.

\begin{lem}\label{Lem_aNK conv}
For $N$ even and $k=0,1,\ldots,N-2$, we have 
\begin{equation}
	\underset{\zeta=1}{\textup{Res}} \, \, a_{N,k}(\zeta) = (-1)^k \sum_{s=0}^{k} \frac{ q_{k,s} }{(k+1-s)!} \frac{ (-1)^s}{N^s},
\end{equation}
where $q_{k,s}$ is defined by \eqref{q_k,s}. 
\end{lem}
\begin{proof}
To compute the residue of $a_{N,k}(\zeta)$ at $\zeta=1$, we first expand $\zeta^{-k-2}$ around $\zeta=1$ as follows:
for $|\zeta-1|<1$, 
$$
\zeta^{-k-2}=\sum_{l=0}^{\infty} (-1)^l \frac{(k+l+1)!}{(k+1)! \, l!} (\zeta-1)^l.
$$
Combining this expansion with the definition \eqref{2F1 Gauss series} of ${}_2F_1$, we obtain 
\begin{align}
\begin{split}
	&\quad \underset{\zeta=1}{\textup{Res}} \, \, a_{N,k}(\zeta) = \frac{1}{(k+1)!} \frac{2}{N^{k+1}} \sum_{l=0}^{N-k-2} (-1)^l \frac{(k+l+1)!}{ \, l!}
	\\
	&=\frac{(-1)^k}{(k+1)!} \Big[ (-1)^k \frac{(k+1)!}{(2N)^{k+1}} + \frac{2}{N^{k}} \frac{ (N-1)!}{(N-k-1)!} {}_2 F_1(1,N+1;N-k;-1) \Big]
	\\
	&=\frac{(-1)^k}{(k+1)!} \Big[ (-1)^k \frac{(k+1)!}{(2N)^{k+1}} + \frac{1}{N^{k}} \frac{ (N-1)!}{(N-k-1)!} {}_2 F_1(1,-k-1;N-k;\tfrac12) \Big].
\end{split}
\end{align}
Here the last identity follows from Pfaff's transformation
\begin{equation} \label{Pfaff trans}
	{}_2F_1(a,b;c;z)=(1-z)^{-a} {}_2F_1(a,c-b;c;\tfrac{z}{z-1}).
\end{equation}
Moreover, it follows from \eqref{2F1 asymp} and 
$$
\frac{\Gamma(-k-1+s)}{\Gamma(-k-1)}=(-1)^s \frac{(k+1)!}{(k+1-s)!}
$$
(where we treat the case $k \in \mathbb{N}$ in the left-hand side as a removable singularity) that 
$$
 \frac{1}{N^{k}} \frac{ (N-1)!}{(N-k-1)!} {}_2 F_1(1,-k-1;N-k;\tfrac12) = \sum_{s=0}^{k+1} q_{k,s} \frac{(k+1)!}{(k+1-s)!} \frac{ (-1)^s}{N^s}.
$$
Therefore we obtain 
\begin{align*}
\underset{\zeta=1}{\textup{Res}} \, \, a_{N,k}(\zeta) &= \frac{1}{(2N)^{k+1}}+ (-1)^k \sum_{s=0}^{k+1} \frac{ q_{k,s} }{(k+1-s)!} \frac{ (-1)^s}{N^s}.
\end{align*}
Finally, using the definition \eqref{q_k,s}, we compute $q_{k,k+1}:$ 
$$
q_{k,k+1}=\underset{z=0}{\textup{Res}} \Big[ \Big( \frac{e^z}{e^z-1} \Big)^{k+2} \frac{2}{1+e^z} \Big] =2^{-k-1}.
$$
We finish the proof.
\end{proof}

Let us write
\begin{align}
\begin{split}
\label{d_s(alpha)}
d_s(\alpha)&:=\frac{(-1)^s }{2} \sum_{k=s}^\infty q_{k,s} \, \frac{(-1)^k}{(k+1-s)!} \frac{(2k-1)!!}{k!} \Big( \frac{\alpha^2}{2} \Big)^{k}.
\end{split}
\end{align}
Then we have the following. 

\begin{lem} \label{Lem_Bessel poly}
We have 
\begin{align}
d_0(\alpha)&= \tfrac{1}{2} e^{-\alpha^2/2} [ I_0( \tfrac{\alpha^2}{2} )+I_1( \tfrac{\alpha^2}{2} ) ],
\\
d_1(\alpha)&=\tfrac14+ \tfrac{1}{8}e^{-\alpha^2/2} [ (\alpha^2-2) I_0( \tfrac{\alpha^2}{2} )-\alpha^2 I_1( \tfrac{\alpha^2}{2} ) ].
\end{align}
\end{lem}

\begin{proof}
Using \eqref{q_k,s}, we first compute
\begin{equation} \label{q k 123}
	q_{k,0}=1, \qquad q_{k,1}=\frac{k+1}{2}.
\end{equation}
By \eqref{I nu}, we have
\begin{align} 
e^{-\alpha^2/2} I_0( \tfrac{\alpha^2}{2} )&=\sum_{k=0}^\infty \frac{(2k-1)!!}{(k!)^2} (-1)^k \Big(\frac{\alpha^2}{2}\Big)^k, \label{I0 exp}
\\
e^{-\alpha^2/2} I_1( \tfrac{\alpha^2}{2} )&=\sum_{k=0}^\infty \frac{(2k-1)!!}{(k-1)!(k+1)!} (-1)^{k+1}\Big(\frac{\alpha^2}{2}\Big)^k. \label{I1 exp}
\end{align}
Then the lemma follows immediately from \eqref{I0 exp} and \eqref{I1 exp}. 
\end{proof}

We are now ready to prove Theorem~\ref{Thm_EN}.

\begin{proof}[Proof of Theorem~\ref{Thm_EN}]
By Lemma~\ref{Prop_gN}, we have
\begin{equation} \label{EN g h}
E_{N,\tau_N}=N \, \sqrt{4-\tfrac{2\alpha^2}{N}} \,\frac{1}{\alpha \sqrt{\pi}} \frac{1}{\sqrt{N}}\,g_N( 1-\tfrac{N}{\alpha^2} )
\end{equation}
and
\begin{align}
	\begin{split}
		\frac{g_N(x)}{\sqrt{N}}&=\sqrt{\frac{\pi}{N}} \,\, \underset{\zeta=1}{\textup{Res}} \, \Big[ \frac{\zeta^{-\frac32} (2-\zeta)^{-1} (1-\zeta x)^{-\frac12}}{ (\zeta-1)^{N-1} } \Big].
	\end{split}
\end{align}

Let us compute the power series expansion of $\frac{1}{\sqrt{N}}g_N( 1-\tfrac{N}{\alpha^2} ).$ 
For this computation, we first expand $(1-\zeta x)^{-\frac12}$ around $\zeta =1$ using the binomial theorem. More precisely, for $x=1-\tfrac{N}{\alpha^2}$, we have
\begin{equation}
	(1-\zeta x)^{-\frac12} = 2\sum_{k=0}^\infty (\zeta-1)^k (\zeta N)^{-k-\frac12} \frac{(2k)!}{(k!)^2 } \Big( \frac{\alpha}{2} \Big)^{2k+1},\quad (|\zeta-1| < |\zeta|\frac{N}{\alpha^2}).
\end{equation}
Therefore we obtain
\begin{equation}
\frac{1}{\sqrt{N}}\,g_N( 1-\tfrac{N}{\alpha^2} )=\sqrt{\pi}\sum_{k=0}^\infty 	\underset{\zeta=1}{\textup{Res}} \, \, a_{N,k}(\zeta) \frac{(2k)!}{(k!)^2 } \Big( \frac{\alpha}{2} \Big)^{2k+1},
\end{equation}
where $a_{N,k}$ is given by \eqref{a_{N,k}}.

By Lemma~\ref{Lem_aNK conv}, we have 
\begin{align*}
&\quad \frac{1}{\alpha \sqrt{\pi}} \frac{1}{\sqrt{N}}\,g_N( 1-\tfrac{N}{\alpha^2} )=\frac12\sum_{k=0}^\infty (-1)^k\frac{(2k-1)!!}{k!} \Big( \frac{\alpha^2}{2} \Big)^{k} \sum_{s=0}^{k} \frac{ q_{k,s} }{(k+1-s)!} \frac{ (-1)^s}{N^s}. 
\end{align*}
Rearranging the terms, for any positive integer $m \ge 2$,
\begin{align}
\frac{1}{\alpha \sqrt{\pi}} \frac{1}{\sqrt{N}}\,g_N( 1-\tfrac{N}{\alpha^2} )
&=\sum_{s=0}^{m-1} \frac{d_s(\alpha)}{N^s}+O(\frac{1}{N^m}),
\end{align}
where $d_s(\alpha)$ is given by \eqref{d_s(alpha)}. Now it follows from \eqref{EN g h} and the binomial expansion 
$$ \sqrt{4-\tfrac{2\alpha^2}{N}}=2-\sum_{s=1}^{m-1} \frac{(2s-3)!!\, \alpha^{2s} }{2^{2s-1} s!} \, \frac{1}{N^s}+O(\frac{1}{N^m}) $$
that 
\begin{equation}
E_{N,\tau_N}=N c(\alpha)+c_0(\alpha)+\frac12+\sum_{l=1}^{m-1} \frac{c_l(\alpha)}{N^l}+O(\frac{1}{N^m}).
\end{equation}
Here $c(\alpha)=2d_0(\alpha)$ and $c_l(\alpha)$'s are given by 
\begin{equation}\label{c0}
	c_0(\alpha):=2d_{1}(\alpha)-\frac{\alpha^2}{2}d_0(\alpha)-\frac12
\end{equation}
and for $l \ge 1$, 
$$
c_l(\alpha)=2d_{l+1}(\alpha)-\alpha \sum_{k=0}^{l} \frac{(2k-1)!! }{ (k+1)! } \Big( \frac{\alpha}{2} \Big)^{2k+1} d_{l-k}(\alpha).
$$
Now Lemma~\ref{Lem_Bessel poly} completes the proof. 

\end{proof}

\section{Distributions of real eigenvalues} \label{Section_distributions}

This section is devoted to proving Theorem~\ref{Thm_density}. 

It was obtained in \cite{MR2430570} that the density $\rho_N\equiv \rho_{N,\tau_N}$ of real eigenvalues can be expressed in terms of the Hermite polynomial
$H_k(x):=(-1)^k e^{x^2} \tfrac{d^k}{dx^k}e^{-x^2}$ as 
\begin{align}\label{rho_N}
	\begin{split}
		\rho_{N}(x)&=\rho_N^1(x)+\rho_{N}^2(x), \qquad \rho_N^{j}=\frac{1}{E_{N,\tau_N}} \textbf{R}_{N}^{j}(x), \qquad (j=1,2)
	\end{split}
\end{align}
where 
\begin{align}
	\textbf{R}_{N}^1(x)&:=\sqrt{\frac{N}{2\pi}} e^{ -\frac{N}{ 1+\tau_N } x^2 } \sum_{k=0}^{N-2} \frac{(\tau_N/2)^k}{k!} H_k( \sqrt{\tfrac{N}{2\tau_N }} x)^2, \label{RN1}
	\\
	\begin{split}
		\textbf{R}_{N}^2(x)&:= \frac{1}{\sqrt{2\pi}} \frac{(\tau_N/2)^{N-\frac32} }{ 1+\tau_N } \frac{ N }{ (N-2)! } e^{ -\frac{N}{2(1+\tau_N)} x^2 } H_{N-1}( \sqrt{\tfrac{N}{2\tau_N }} x ) 
		\\
		&\quad \times \int_{0}^{x} e^{ -\frac{N}{2(1+\tau_N)} u^2 } H_{N-2} ( \sqrt{\tfrac{N}{2\tau_N }} u )\,du. \label{RN2}
	\end{split}
\end{align}
Here we use a different normalization so that the limiting empirical distribution has a compact support, cf. \eqref{Ellipse}. 

The overall strategy to derive the large-$N$ limit \eqref{Goal2} of $\rho_N$ is as follows:
\begin{itemize}
 \item we first compute the large-$N$ limit of $\rho_N(0)$ (Lemma~\ref{Lem_density 0}) by virtue of some basic properties of hypergeometric functions;
 \smallskip 
 \item using a version of the Christoffel-Darboux formula (Lemma~\ref{Lem_Hermite}) we obtain an integral representation of $\rho_N^1$ (Lemma~\ref{Prop_rhoN1 rep}), where the previous step is utilized to determine the integration constant;
 \smallskip 
 \item we derive the large-$N$ limit of the integral representation using the Plancherel-Rotach asymptotic formula for Hermite polynomials \eqref{PR in} and some asymptotic analysis on the associated oscillatory integrals (Lemma~\ref{Lem_Oscill}).
\end{itemize}

We begin with the following lemma, which gives Theorem~\ref{Thm_density} for the specific value $x=0.$

\begin{lem} \label{Lem_density 0}
We have 
\begin{align}
	\lim_{ N \to \infty } 	\rho_{N}(0) = 	\frac{1}{c(\alpha)}\frac{1}{2\alpha \sqrt{\pi}} \erf(\alpha).
\end{align}	
\end{lem}
\begin{proof}
Using Hermite numbers $$H_k(0)=\frac{\sqrt{\pi}}{\Gamma(\frac{1-k}{2})} \, 2^k$$
and \eqref{2F1 Gauss series}, we have 
\begin{equation*}
\rho_{N}(0)=\frac{1}{E_{N,\tau_N}} \Big(\frac{N}{2\pi}\Big)^{\frac12} \Big[ (1-\tau_N^2)^{-\frac12}-\tau_N^N \, \frac{ \Gamma(\frac{N+1}{2})}{\sqrt{\pi}} {}_2\textup{\textbf{F}}_1( 1,\tfrac{N+1}{2};\tfrac{N}{2}+1;\tau_N^2 ) \Big]. 
\end{equation*}
By \eqref{Pfaff trans}, we have
$$
{}_2\textup{\textbf{F}}_1( 1,\tfrac{N+1}{2};\tfrac{N}{2}+1;\tau_N^2 ) =(1-\tau_N^2) ^{-1} {}_2\textup{\textbf{F}}_1( 1,\tfrac12;\tfrac{N}{2}+1; \tfrac{\tau_N^2}{\tau_N^2-1} ).
$$
This leads to 
$\rho_{N}(0)=a_N-b_N,$
where
\begin{align}
	a_N&=\frac{1}{E_{N,\tau_N}} \Big(\frac{N}{2\pi(1-\tau_N^2)}\Big)^{\frac12}, 
	\\
	b_N&=\frac{1}{E_{N,\tau_N}} \Big(\frac{N}{2\pi}\Big)^{\frac12} \frac{\tau_N^N}{1-\tau_N^2} \, \frac{ \Gamma(\frac{N+1}{2})}{\sqrt{\pi}} {}_2\textup{\textbf{F}}_1( 1,\tfrac12;\tfrac{N}{2}+1; \tfrac{\tau_N^2}{\tau_N^2-1} ) . 
\end{align}

Note here that by Theorem~\ref{Thm_EN}, we have 
\begin{equation}
\lim_{N\to\infty} a_N = \frac{1}{c(\alpha)} \frac{1}{2\alpha \sqrt{\pi }}. 
\end{equation}
Therefore it suffices to show that
\begin{equation}
\lim_{N\to\infty}	b_N =\frac{1}{c(\alpha)} \frac{1}{2\alpha \sqrt{\pi }} \erfc(\alpha) .
\end{equation}
Notice also that Theorem~\ref{Thm_EN} gives rise to 
$$
b_N \sim \frac{1}{c(\alpha)} \frac{1}{2\alpha \sqrt{\pi}} \Big( \frac{N}{2\pi} \Big)^{\frac12} \Gamma( \tfrac{N+1}{2} ) \frac{e^{-\alpha^2}}{\alpha} {}_2\textup{\textbf{F}}_1( 1,\tfrac12;\tfrac{N}{2}+1; \tfrac{\tau_N^2}{\tau_N^2-1} ) .
$$
Thus we need to show 
\begin{equation}
\lim_{ N \to \infty } \Big( \frac{N}{2\pi} \Big)^{\frac12} \Gamma( \tfrac{N+1}{2} ) {}_2\textup{\textbf{F}}_1( 1,\tfrac12;\tfrac{N}{2}+1; -\tfrac{N}{2\alpha^2} ) =\alpha \, e^{\alpha^2} \erfc(\alpha). 
\end{equation}

It follows from \cite[(15.4.6), (15.8.2)]{olver2010nist} that 
\begin{align*}
&\quad \Big( \frac{N}{2\pi} \Big)^{\frac12} \Gamma( \tfrac{N+1}{2} ) {}_2\textup{\textbf{F}}_1( 1,\tfrac12;\tfrac{N}{2}+1; -\tfrac{N}{2\alpha^2} )
\\
&= \alpha \Big( 1+\frac{2\alpha^2}{N} \Big)^{ \frac{N-1}{2} }
 -\sqrt{ \frac{2}{N} } \frac{ \Gamma( \tfrac{N+1}{2} ) }{ \Gamma(\frac{N}{2}) } \alpha^2 {}_2\textup{\textbf{F}}_1( 1-\tfrac{N}{2},1;\tfrac{3}{2}; -\tfrac{2\alpha^2}{N} ).
\end{align*}
Then by Stirling's formula, we have 
\begin{equation}
\lim_{ N \to \infty } \Big( \frac{N}{2\pi} \Big)^{\frac12} \Gamma( \tfrac{N+1}{2} ) {}_2\textup{\textbf{F}}_1( 1,\tfrac12;\tfrac{N}{2}+1; -\tfrac{N}{2\alpha^2} ) = \alpha \, e^{\alpha^2}-\alpha^2 \lim_{ N \to \infty } {}_2\textup{\textbf{F}}_1( 1-\tfrac{N}{2},1;\tfrac{3}{2}; -\tfrac{2\alpha^2}{N} ).
\end{equation}
Using the Euler integral formula (see e.g. \cite[Eq.(15.6.1)]{olver2010nist}): for $c>b>0$, 
$$
{}_2\textup{\textbf{F}}_1(a,b,c,z)=\frac{1}{\Gamma(b)\Gamma(c-b)} \int_0^1 \frac{ t^{b-1} (1-t)^{c-b-1} }{ (1-z t)^a }\,dt,
$$
we have
\begin{align*}
\lim_{N\to \infty} {}_2\textup{\textbf{F}}_1( 1-\tfrac{N}{2},1;\tfrac{3}{2}; -\tfrac{2\alpha^2}{N} )&=\lim_{N\to \infty} \frac{1}{\sqrt{\pi}} \int_0^1 \frac{1}{\sqrt{1-t}} \Big( 1+\frac{2\alpha^2}{N}t \Big)^{ \frac{N}{2}-1 }\,dt
\\
&= \frac{1}{\sqrt{\pi}} \int_0^1 \frac{e^{\alpha^2 t}}{\sqrt{1-t}} \,dt=\frac{ e^{\alpha^2} }{ \alpha } \erf(\alpha).
\end{align*}
This completes the proof. 
\end{proof}

We shall use the following version of the Christoffel--Darboux identity for Hermite polynomials, see \cite[Lemma 4.1]{AB20} and \cite[Proposition 2.3]{lee2016fine}.

\begin{lem}\label{Lem_Hermite}
	For any $\tau \in (0,1]$, let
	\begin{align} \label{F tau}
		F_N(x):=\sum_{k=0}^{N-2} \frac{(\tau/2)^k}{k!} H_k(x) ^2.
	\end{align}
	Then for any $N \ge 2$, we have 
	\begin{align} \label{Mehler fin N}
		F'_N(x)=\frac{4\tau\,x}{1+\tau}F_N(x)-\frac{4(\tau/2)^{N-1}}{1+\tau} \frac{ H_{N-2}(x) H_{N-1}(x)}{(N-2)!}.
	\end{align}
\end{lem}

For the reader's convenience, let us present the proof of Lemma~\ref{Lem_Hermite} here.

\begin{proof}[Proof of Lemma~\ref{Lem_Hermite}]
For $N=2$, it follows from $H_0(x)=1,$ $H_1(x)=2x$ that the identity \eqref{Mehler fin N} trivially holds. 
By the differentiation rule 
$$
H_{N-1}'(x)=2(N-1)H_{N-2}(x)
$$
and the three-term recurrence relation
$$
H_N(x)=2xH_{N-1}(x)-H_{N-1}'(x)
$$
of the Hermite polynomials, we have
$$
\tau x H_{N-1}(x) +(N-1) H_{N-2}(x)-\frac{1+\tau}{2}H_{N-1}'(x) = \frac{\tau}{2} H_N(x). 
$$ 
Using this, the induction argument gives 
\begin{align*}
F_{N+1}'(x)&=F_N'(x)+\frac{(\tau/2)^{N-1}}{(N-1)!} \Big(H_{N-1}(x)^2\Big)' 
\\
&=\frac{4\tau\,x}{1+\tau}F_{N+1}(x)
-\frac{4(\tau/2)^{N}}{1+\tau}\frac{H_{N-1}(x) H_N(x)}{(N-1)!},
\end{align*}
which completes the proof. 
\end{proof}

Using Lemma~\ref{Lem_Hermite}, we obtain an integral representation of $\rho_N^1$, a key ingredient for the latter asymptotic analysis. 

\begin{lem} \label{Prop_rhoN1 rep}
For any $\tau_N \in (0,1]$, we have 
\begin{align}
	\begin{split}
		\rho_{N}^1(x)=\rho_{N}^1(0)&-\frac{1}{E_{N,\tau_N} }\sqrt{\frac{2}{\pi}} \frac{(\tau_N/2)^{N-\frac32}}{1+\tau_N} \frac{N}{(N-2)!}
		\\
		&\quad \times \int_{0}^{x} e^{ -\frac{N}{ 1+\tau_N } u^2 } H_{N-2}( \sqrt{\tfrac{N}{2\tau_N }} u)H_{N-1}( \sqrt{\tfrac{N}{2\tau_N }} u)\,du.
	\end{split}
\end{align}
\end{lem} 
\begin{proof}
By \eqref{RN1}, the function $\textbf{R}_{N}^1$ is expressed in terms of the function $F_N$ in \eqref{F tau} as
$$
\textbf{R}_{N}^1(x)=\sqrt{\frac{N}{2\pi}} e^{ -\frac{N}{ 1+\tau_N } x^2 } F_N( \sqrt{\tfrac{N}{2\tau_N }} x).
$$
Differentiating the above and using Lemma~\ref{Lem_Hermite}, we obtain 
\begin{align*}
(\textbf{R}_{N}^1)'(x) = -\sqrt{\frac{2}{\pi}} \frac{(\tau_N/2)^{N-\frac32}}{1+\tau_N} \frac{N}{(N-2)!} e^{ -\frac{N}{ 1+\tau_N } x^2 } H_{N-2}( \sqrt{\tfrac{N}{2\tau_N }} x)H_{N-1}( \sqrt{\tfrac{N}{2\tau_N }} x).
\end{align*}
Integrating the above, we have 
\begin{align}
\begin{split}
\textbf{R}_{N}^1(x)=\textbf{R}_{N}^1(0)&-\sqrt{\frac{2}{\pi}} \frac{(\tau_N/2)^{N-\frac32}}{1+\tau_N} \frac{N}{(N-2)!}
\\
& \quad \times \int_{0}^{x} e^{ -\frac{N}{ 1+\tau_N } u^2 } H_{N-2}( \sqrt{\tfrac{N}{2\tau_N }} u)H_{N-1}( \sqrt{\tfrac{N}{2\tau_N }} u)\,du, 
\end{split}
\end{align}
which completes the proof. 
\end{proof}

We will need the following lemma on the boundedness of certain oscillatory integrals.
 
\begin{lem} \label{Lem_Oscill}
As $N \to \infty$, we have
\begin{equation} \label{Oscill lem1}
\int_{0}^{x} \frac{e^{\alpha^2 u^2/8} }{(4-u^2)^{\frac14}} \cos\Big\{ \tfrac{N+\alpha^2}{4}u\sqrt{4-u^2} -N\arccos( \tfrac{u}{2} ) -\tfrac{3}{2} \arcsin( \tfrac{u}{2}) \Big\} \,du=O(\frac1N)
\end{equation}
and 
\begin{equation} \label{Oscill lem2}
\int_{0}^x \frac{e^{\alpha^2 u^2/4}}{\sqrt{4-u^2}}\sin\Big\{ \tfrac{N+\alpha^2}{2} u\sqrt{4-u^2} -2N \arccos( \tfrac{u}{2} )-2 \arcsin( \tfrac{u}{2} ) \Big\} \,du=O(\frac1N)
\end{equation}
uniformly on compact subsets of $(-2,2).$
\end{lem}

\begin{proof}
We present the proof of the first assertion \eqref{Oscill lem1} as the other one \eqref{Oscill lem2} follows from similar computations. 
Let us write 
\begin{align*}
\phi(u)&:=\frac{u \sqrt{4-u^2} }{4} -\arccos(\tfrac{u}{2}),
\\
a(u)&:= \frac{e^{\alpha^2 u^2/8} }{(4-u^2)^{\frac14}} \exp\Big\{ i(\tfrac{\alpha^2}{4}u \sqrt{4-u^2}-\tfrac32 \arcsin(\tfrac{u}{2}) ) \Big\}.
\end{align*}
Then the desired asymptotic \eqref{Oscill lem1} is equivalent to 
\begin{equation}\label{Oscill lem1 var}
 \Re \int_0^x a(u)e^{iN \phi(u)}\,du= O(\frac1N),
\end{equation}
where the $O(\frac1N)$ term is uniform on compact subsets of $(-2,2).$

Since
$$
\phi'(u)=\frac{\sqrt{4-u^2}}{2},
$$
the function $\phi(u)$ has critical points only at $u=\pm 2.$ Thus we have that for $x \in (0,2),$ 
\begin{align*}
\int_0^x a(u)e^{iN \phi(u)}\,du&=\frac{1}{iN} \int_0^x \frac{a(u)}{\phi'(u)} \, \frac{d}{du}e^{iN \phi(u)}\,du
\\
&= \frac{1}{iN} \Big( \frac{a(u)}{\phi'(u)} e^{iN\phi(u)}\Big|_{u=0}^{u=x} - \int_0^x \frac{d}{du}\Big(\frac{a(u)}{\phi'(u)}\Big) \, e^{iN \phi(u)}\,du \Big) 
\end{align*}
since $\phi$ has no critical point on $[0,x]$.
Therefore we obtain 
\begin{equation}
\Big| \int_0^x a(u)e^{iN \phi(u)}\,du \Big| \le \frac1N \Big( \Big\| \frac{a}{\phi'} \Big\|_{ L^\infty[0,x] }+ \Big\| \Big(\frac{a}{\phi'}\Big)' \Big\|_{ L^1[0,x] } \Big),
\end{equation}
which leads to \eqref{Oscill lem1 var}. 
One can treat the other case $x\in(-2,0)$ in the same way.
\end{proof}

Let us recall the Plancherel-Rotach asymptotic formula (see e.g., \cite[Chapter 7]{MR1677884} or \cite[Theorem 5]{MR2364955}): for $\theta \in (-\pi/2,\pi/2)$,
\begin{equation} \label{PR in}
	H_N(\sqrt{2N} \sin \theta )\sim \sqrt{ \tfrac{2}{\cos \theta} } (2N)^{\frac{N}{2}} e^{ -(\frac12-\sin^2 \theta) N }
 \cos \Big\{ N[ -\tfrac{\pi}{2}+\theta+ \tfrac12 \sin 2\theta ]+\tfrac{\theta}{2} \Big\}.
\end{equation}
We are now ready to prove Theorem~\ref{Thm_density}.

\begin{proof}[Proof of Theorem~\ref{Thm_density}]
We prove the theorem by using the following steps:
\begin{itemize}
	\item for any $x \in (-2,2),$ 
	\begin{equation} \label{rhoN2 conv}
		\rho_N^2(x) \to 0 \quad \text{as } N \to \infty; 
	\end{equation}
 \item for any $x \in (-2,2),$ 
 \begin{align} \label{rhoN1 conv}
 	\rho_{N}^1(x) \to	\frac{1}{c(\alpha)} \frac{1}{2\alpha \sqrt{\pi}} \erf( \tfrac{\alpha}{2} \sqrt{4-x^2} ) \quad \text{as }N\to\infty. 
 \end{align}	
\end{itemize}
Combining \eqref{rhoN2 conv} and \eqref{rhoN1 conv}, as $N \to \infty,$
\[
\rho_N(x) \to \rho_{\alpha}^w(x), \quad \text{for }|x|<2. 
\]
Since both $\rho_N$ and $\rho_{\alpha}^w$ are probability density functions, we have 
\[
\lim_{ N \to \infty } \rho_N(x) = 0, \quad \text{for }|x|\ge 2,
\]
which proves Theorem~\ref{Thm_density}. 
Therefore it suffices to show \eqref{rhoN2 conv} and \eqref{rhoN1 conv}.

Let us first show \eqref{rhoN2 conv}. 
By Theorem~\ref{Thm_EN}, we have 
\begin{equation} \label{fact rho asy}
\frac{1}{E_{N,\tau_N} }\sqrt{\frac{2}{\pi}} \frac{(\tau_N/2)^{N-\frac32}}{1+\tau_N} \frac{N}{(N-2)!} \sim \frac{1}{c(\alpha)} \frac{2^{-N+1}}{\sqrt{\pi}} \frac{e^{-\alpha^2}}{ (N-2)! }.
\end{equation}
Then by \eqref{RN2},
\begin{align}
	\begin{split}
		\rho_N^2(x)&\sim \frac{e^{-\alpha^2}}{c(\alpha)} \frac{2^{-N}}{\sqrt{\pi}} \frac{e^{ -\frac{N}{2(1+\tau_N)} x^2 }}{ (N-2)! } H_{N-1}( \sqrt{\tfrac{N}{2\tau_N }} x ) 
	\int_{0}^{x} e^{ -\frac{N}{2(1+\tau_N)} u^2 } H_{N-2} ( \sqrt{\tfrac{N}{2\tau_N }} u )\,du.
	\end{split}
\end{align}
We now use \eqref{PR in} to obtain
\begin{align} \label{PR n-1}
\begin{split}
 H_{N-1}( \sqrt{\tfrac{N}{2\tau_N }} x ) &\sim -\frac{2}{ (4-x^2)^{\frac14} } (2N)^{ \frac{N-1}{2} } e^{ \frac{N}{2(1+\tau_N)} x^2 }\, e^{-\tfrac{N}{2}+\tfrac{\alpha^2 x^2}{8} }
\\
& \times \sin\Big\{ \tfrac{N+\alpha^2}{4}x\sqrt{4-x^2} -N\arccos( \tfrac{x}{2} ) -\tfrac{1}{2} \arcsin( \tfrac{x}{2}) \Big\}
\end{split}
\end{align}
and 
\begin{align} \label{PR n-2}
\begin{split}
 H_{N-2} ( \sqrt{\tfrac{N}{2\tau_N }} u ) &\sim -\frac{2}{ (4-u^2)^{\frac14} } (2N)^{ \frac{N-2}{2} }e^{ \frac{N}{2(1+\tau_N)} u^2 } \, e^{ -\tfrac{N}{2}+\tfrac{\alpha^2 u^2}{8} } 
\\
& \times \cos\Big\{ \tfrac{N+\alpha^2}{4}u\sqrt{4-u^2} -N\arccos( \tfrac{u}{2} ) -\tfrac{3}{2} \arcsin( \tfrac{u}{2}) \Big\}.
\end{split}
\end{align}
Combining \eqref{PR n-1} and \eqref{PR n-2} with Stirling's formula, we obtain 
\begin{align*}
	\begin{split}
	 \rho_N^2(x)&\sim \frac{e^{-\alpha^2}}{c(\alpha)} \frac{1}{\pi} \frac{e^{\alpha^2 x^2/8} }{(4-x^2)^{\frac14}} \sin\{ \tfrac{N+\alpha^2}{4}x\sqrt{4-x^2} -N\arccos( \tfrac{x}{2} ) -\tfrac{1}{2} \arcsin( \tfrac{x}{2}) \Big\}
		\\
		& \times \int_{0}^{x} \frac{e^{\alpha^2 u^2/8} }{(4-u^2)^{\frac14}} \cos\Big\{ \tfrac{N+\alpha^2}{4}u\sqrt{4-u^2} -N\arccos( \tfrac{u}{2} ) -\tfrac{3}{2} \arcsin( \tfrac{u}{2}) \Big\} \,du.
	\end{split}
\end{align*}
Now the claim \eqref{rhoN2 conv} follows from the first assertion \eqref{Oscill lem1} of Lemma~\ref{Lem_Oscill}.

Next, we show \eqref{rhoN1 conv}. For $x\in(-2,2)$, let us write 
\begin{equation}
	h_N(x):=\frac{2^{-N}}{(N-2)!}\int_{0}^{x} e^{ -\frac{N}{ 1+\tau_N } u^2 } H_{N-2}( \sqrt{\tfrac{N}{2\tau_N }} u)H_{N-1}( \sqrt{\tfrac{N}{2\tau_N }} u)\,du
\end{equation}
and 
\begin{equation}
	h(x):= \frac{1}{4\alpha}\, e^{\alpha^2} \, \Big( \erfc( \tfrac{\alpha}{2} \sqrt{4-x^2} )-\erfc(\alpha) \Big).
\end{equation}
Combining Lemma~\ref{Prop_rhoN1 rep} with \eqref{fact rho asy}, the claim \eqref{rhoN1 conv} is equivalent to the convergence
\begin{align}\label{Goal2 var2}
	h_N(x) \to h(x) \quad \text{as }N\to \infty. 
\end{align}

To show this, we first investigate the asymptotic behavior of ${h_N'(x)}/{h'(x)}$ as $N\to\infty$.
By differentiation, 
\begin{equation}
h_N'(x)=\frac{2^{-N}}{(N-2)!} e^{ -\frac{N}{ 1+\tau_N } x^2 } H_{N-2}( \sqrt{\tfrac{N}{2\tau_N }} x)H_{N-1}( \sqrt{\tfrac{N}{2\tau_N }} x)
\end{equation}
and 
\begin{equation}
h'(x)=\frac{1}{4\sqrt{\pi}} \frac{x}{\sqrt{4-x^2}} e^{\alpha^2 x^2/4} .
\end{equation}
Again we use \eqref{PR n-1} and \eqref{PR n-2} to derive
\begin{align}
	\frac{h_N'(x)}{h'(x)} \sim 1+\frac{2}{x}\sin\Big\{ \tfrac{N+\alpha^2}{2} x\sqrt{4-x^2} -2N \arccos( \tfrac{x}{2} )-2 \arcsin( \tfrac{x}{2} ) \Big\} .
\end{align}
Then \eqref{Goal2 var2} follows from the second assertion \eqref{Oscill lem2} of Lemma~\ref{Lem_Oscill}.
We finish the proof.
\end{proof}

\section{Variance of the number of real eigenvalues}\label{Section_Variance}

In this section, we prove Theorem~\ref{Thm_Variance}. 
Let us define
\begin{equation} \label{S_N}
	S_N(x,y):=S_{N}^{1}(x,y)+S_{N}^{2}(x,y),
\end{equation}
where 
\begin{align}
	\begin{split}
	S_{N}^{1}(x,y)&:=\frac{1}{\sqrt{2\pi}}e^{-\frac{x^2+y^2}{2(1+\tau_N)}} \sum_{k=0}^{N-2} \frac{(\tau_N/2)^k}{k!} H_k( \tfrac{x}{\sqrt{2\tau_N}} ) H_k( \tfrac{y}{\sqrt{2\tau_N}} ), \label{SN1}
		\end{split}
	\\
	\begin{split}
	S_{N}^{2}(x,y)&:=\frac{1}{\sqrt{2\pi}} \frac{(\tau_N/2)^{N-\frac32} }{1+\tau_N} \frac{1 }{(N-2)!} e^{-\frac{y^2}{2(1+\tau_N)} } H_{N-1} ( \tfrac{y}{\sqrt{2\tau_N}} ) 
	\\
	&\quad \times \int_{0}^x e^{ -\frac{u^2}{2(1+\tau_N)} } H_{N-2} ( \tfrac{u}{\sqrt{2\tau_N}} )\,du. \label{SN2}
		\end{split}
\end{align}
The function $S_N$ is used to form $2 \times 2$ matrix-valued kernel of the associated Pfaffian point process, see \cite{MR2430570}. 

By definition, we have
\begin{equation}
	S_N(x,x)=\frac{E_{N,\tau_N} }{\sqrt{N}}\rho_N( \tfrac{x}{\sqrt{N}} ), \qquad 	E_{N,\tau_N}=\int_{\R} S_N(x,x)\,dx.
\end{equation}
The variance $V_{N,\tau_N}$ is expressed in terms of one- and two-point correlation functions of the real eigenvalues, which in turn 
\begin{align} \label{VN}
	\begin{split}
		V_{N,\tau_N}=2E_{N,\tau_N} -2\int_{\R^2} S_N(x,y)S_N(y,x)\,dx\,dy,
	\end{split}
\end{align}
see \cite{forrester2007eigenvalue}. 

By Theorem~\ref{Thm_EN}, we have $E_{N,\tau_N} \sim N c(\alpha).$ Observe also that the functions $r(\alpha)$ and $c(\alpha)$ are related as
$$
r(\alpha)= 2-2 \, \frac{c(\sqrt{2}\alpha)}{c(\alpha)}.
$$
Then by \eqref{VN}, it suffices to show that 
\begin{align}\label{VN Goal1}
\begin{split}
\lim_{ N \to \infty } \frac{1}{N}\int_{\R^2} S_N(x,y) S_N(y,x)\,dx\,dy = c (\sqrt{2}\alpha).
\end{split}
\end{align}
 
Notice that the integral in \eqref{SN2} is the same as the one in \eqref{RN2}. 
Thus we can use \eqref{PR n-1}, \eqref{PR n-2}, and Lemma~\ref{Lem_Oscill} to obtain that for any $x,y \in (-2\sqrt{N},2\sqrt{N})$, 
$$
S_N^2(x,y)=o(\sqrt{N} ), \qquad S_N^2(y,x)=o(\sqrt{N} ).
$$
Therefore the term $S_N^2$ does not contribute to the leading order asymptotic of \eqref{VN Goal1}, namely
\begin{equation}\label{Sint S1int}
 \frac{1}{N}\int_{\R^2} S_N(x,y) S_N(y,x)\,dx\,dy \sim \frac{1}{N}\int_{\R^2} S_N^1(x,y)^2\,dx\,dy .
\end{equation}
By \eqref{I nu} and \eqref{c(alpha)}, all we need to show is the convergence
\begin{equation} \label{VN Goal11}
\lim_{ N \to \infty } \frac{1}{N}\int_{\R^2} S_N^1(x,y)^2\,dx\,dy 
=\sum_{k=0}^\infty (-1)^k\frac{(2k-1)!!}{k!(k+1)!} \alpha^{2k}.
\end{equation}

The overall strategy of the proof of \eqref{VN Goal11} is similar to that in Section~\ref{Section_expectation} albeit the computations are fairly more complicated. For the reader's convenience, let us summarize the steps of the proof as follows: 
\begin{itemize}
 \item we first obtain an expression of $\frac{1}{N}\int_{\R^2} S_N^1(x,y)^2\,dx\,dy$ which is given in terms of a \emph{double summation} of hypergeometric functions (Lemma~\ref{Lem_SN1 int});
 \smallskip
 \item we decompose the double summation in the previous step into finite number of single summations (see \eqref{INn}, \eqref{RN1N decomp}) and obtain a \emph{double contour} integral representation of each of the single summation (Lemma~\ref{Lem_INn power});
 \smallskip
 \item we compute the power series expansion of each of the single summation by the residue calculus (Lemma~\ref{Lem_aNkn residue});
 \smallskip
 \item using a combinatorial identity (Lemma~\ref{Lem_Com iden}) we compile all the contributions from each of the single summation in the large-$N$ limit. 
\end{itemize}
See also Remark~\ref{Rmk_residue} below for a motivation of this approach using the residue calculus. 

We begin with the following lemma.

\begin{lem} \label{Lem_SN1 int}
We have 
$$
\frac 1N \int_{\R^2} S_{N}^1(x,y)^2\,dx\,dy=\RN{1}_N(\alpha)+\RN{2}_N(\alpha),
$$
where 
\begin{align*}
	\RN{1}_N(\alpha)&:=\frac1N	\frac{\pi}{2} \frac{1+\tau_N}{1-\tau_N}\sum_{l,m=0}^{N/2-1} \frac{ {}_2\textbf{\textup{F}}_1(l-m+\tfrac{1}{2},m-l+\tfrac12 ;-l-m+\tfrac{1}{2};-\tfrac{\tau_N}{1-\tau_N})^2 }{ (2l)! (2m)!},
	\\
	\RN{2}_N(\alpha)&:=\frac1N\frac{\pi}{2} \frac{1+\tau_N}{1-\tau_N}\sum_{l,m=1}^{N/2-1} \frac{ {}_2\textbf{\textup{F}}_1(l-m+\tfrac{1}{2},m-l+\tfrac12 ;-l-m+\tfrac{3}{2};-\tfrac{\tau_N}{1-\tau_N})^2 }{ (2l-1)! (2m-1)!}.
\end{align*}

\end{lem}

\begin{proof}
By \eqref{SN1}, we have
$$
S_{N,1}(x,y)^2=\frac{1}{2\pi}e^{-\frac{x^2+y^2}{1+\tau_N}} \sum_{j,k=0}^{N-2} \frac{(\tau_N/2)^{j+k}}{j!k!} H_j( \tfrac{x}{\sqrt{2\tau_N}} )H_k( \tfrac{x}{\sqrt{2\tau_N}} ) H_j( \tfrac{y}{\sqrt{2\tau_N}} ) H_k( \tfrac{y}{\sqrt{2\tau_N}} ).
$$
Observe here that for $j+k$ odd, we have 
$$
\int_\R e^{-\frac{x^2}{1+\tau_N}} H_j( \tfrac{x}{\sqrt{2\tau_N}} )H_k( \tfrac{x}{\sqrt{2\tau_N}} ) \,dx=0.
$$
For $j+k$ even, we use the formula (7.374-5) in \cite{gradshteyn2014table} to derive 
\begin{align}
	\begin{split}
&\quad \int_\R e^{-\frac{x^2}{1+\tau_N}} H_j( \tfrac{x}{\sqrt{2\tau_N}} )H_k( \tfrac{x}{\sqrt{2\tau_N}} ) \,dx=\sqrt{2\tau_N} \int_\R e^{-\frac{2\tau_N }{1+\tau_N}u^2} H_j( u )H_k( u ) \,du
\\
&=\sqrt{2\tau_N} \, 2^{ \frac{j+k-1}{2} } ( \tfrac{\tau_N}{1+\tau_N} )^{ -\frac{j+k+1}{2} } (\tfrac{1-\tau_N}{1+\tau_N} )^{ \frac{j+k}{2} } \Gamma(\tfrac{j+k+1}{2}) {}_2F_1(-j,-k;\tfrac{1-j-k}{2};-\tfrac{\tau_N}{1-\tau_N})
\\
&=( \tfrac{1+\tau_N}{1-\tau_N} )^{\frac12} ( \tfrac{\tau_N}{2} )^{ -\frac{j+k}{2} } \Gamma(\tfrac{j+k+1}{2}) {}_2F_1(\tfrac{j-k+1}{2},\tfrac{k-j+1}{2};\tfrac{1-j-k}{2};-\tfrac{\tau_N}{1-\tau_N}).
	\end{split}
\end{align}
Here the last identity follows from Euler's transformation (see e.g., \cite[Eq.(15.8.1)]{olver2010nist})
\begin{equation} \label{Euler trans}
{}_2F_1(a,b;c;z)=(1-z)^{c-a-b} {}_2F_1(c-a,c-b;c;\tfrac{z}{z-1}).
\end{equation} 
Then by \eqref{Gamma reflection}, we obtain
\begin{equation} \label{SN1 int}
\int_{\R^2} S_{N,1}(x,y)^2\,dx\,dy= \frac{\pi}{2} \frac{1+\tau_N}{1-\tau_N}\sum_{\substack{j+k: even \\ 0 \le j,k \le N-2}} \frac{ {}_2\textbf{F}_1(\tfrac{j-k+1}{2},\tfrac{k-j+1}{2};\tfrac{1-j-k}{2};-\tfrac{\tau_N}{1-\tau_N})^2 }{ j! k!},
\end{equation}
which completes the proof. 
\end{proof}

\begin{rmk}\label{Rmk_residue}
Let us pause here to briefly explain the reason we derive the asymptotics of \eqref{SN1 int} using the residue calculus rather than direct computations. For this, note that 
\begin{equation} \label{SN1 int v2}
\frac{1}{N}\int_{\R^2} S_{N,1}(x,y)^2\,dx\,dy \sim \frac{\pi}{\alpha^2}\sum_{\substack{j+k: even \\ 0 \le j,k \le N-2}} \frac{ {}_2\textbf{F}_1(\tfrac{j-k+1}{2},\tfrac{k-j+1}{2};\tfrac{1-j-k}{2};-\tfrac{\tau_N}{1-\tau_N})^2 }{ j! k!}.
\end{equation}
As a function of $\alpha$, the right-hand side of the equation \eqref{SN1 int v2} is indeed an even polynomial of degree $4N-8.$ More precisely, we have 
\begin{equation} \label{SN1 int poly}
 \frac{\pi}{\alpha^2}\sum_{j,k=0}^{N-2} \frac{ {}_2\textbf{F}_1(\tfrac{j-k+1}{2},\tfrac{k-j+1}{2};\tfrac{1-j-k}{2};-\tfrac{\tau_N}{1-\tau_N})^2 }{ j! k!} = \sum_{l=0}^{2N-4} b_{N,l} \, \alpha^{2l},
\end{equation}
where 
\begin{equation} \label{bNl}
b_{N,l}= \frac{(-1)^l}{\pi N^{l+1}} \sum_{\substack{j+k: even \\ 0 \le j,k \le N-2}} \sum_{\substack{n+m =l \\ 0 \le n \le j \\ 0 \le m \le k}} \frac{\Gamma( \tfrac{k-j+1}{2}+n)\Gamma( \tfrac{j-k+1}{2}+m)}{ \Gamma(k-j+1+n)\Gamma(j-k+1+m) } \binom{j}{n} \binom{k}{m}.
\end{equation}
The identity \eqref{SN1 int poly} follows from well-known properties of hypergeometric function which include Pfaff's transform \eqref{Pfaff trans}, a linear transform \cite[Eq.(15.8.7)]{olver2010nist} and an evaluation in terms of polynomials \cite[Eq.(15.2.4)]{olver2010nist} as well as basic properties of Gamma function such as the reflection formula \eqref{Gamma reflection}. We leave this verification to interested readers. 
Using \eqref{SN1 int poly}, one may prove Theorem~\ref{Thm_Variance} by analyzing the large-$N$ limit of \eqref{bNl}. Nevertheless, it requires some non-trivial combinatorial identities to evaluate the multiple summations. On the other hand, the residue calculus we will perform plays a role in simplifying the multiple summations in \eqref{bNl}, which makes the evaluation easier. The meaning of this will become clear below. 
\end{rmk}

We shall now focus on the asymptotic analysis for $\RN{1}_N$ as the other case follows along the same line. 
By \eqref{tau wH}, we have 
$$
	\RN{1}_N(\alpha) \sim \frac{\pi}{\alpha^2 }\sum_{l,m=0}^{N/2-1} \frac{ {}_2\textbf{F}_1(l-m+\tfrac{1}{2},m-l+\tfrac12 ;-l-m+\tfrac{1}{2};-\tfrac{N}{\alpha^2}+1)^2 }{ (2l)! (2m)!}.
$$
For each non-negative integer $n$, let us write 
\begin{equation}\label{INn}
\RN{1}_{N,n}(\alpha):=\frac{\pi}{\alpha^2 }\sum_{l=n}^{N/2-1} \frac{ {}_2\textbf{F}_1(n+\tfrac{1}{2},-n+\tfrac12 ;-2l+n+\tfrac{1}{2};-\tfrac{N}{\alpha^2}+1)^2 }{ (2l)! (2l-2n)!}.
\end{equation}
Then by letting $l=m-n$, one can see that $\RN{1}_N$ can be asymptotically decomposed as follows: 
\begin{equation}\label{RN1N decomp}
\RN{1}_N(\alpha)\sim \RN{1}_{N,0}(\alpha)+2 \sum_{n=1}^{N/2-1} \RN{1}_{N,n}(\alpha). 
\end{equation}
In the sequel, we shall use the shorthand notation 
$$\underset{\zeta,\eta=1}{\text{Res}}[f(\zeta,\eta)]:= \underset{\eta=1}{\text{Res}}[ \underset{\zeta=1}{\text{Res}}[f(\zeta,\eta)]].$$
Then we obtain the power series expansion of $\RN{1}_{N,n}$ in terms of a double contour integral. 
\begin{lem}\label{Lem_INn power}
We have 
\begin{equation}\label{RN1Nn}
\RN{1}_{N,n}(\alpha)=\frac{1}{2}\sum_{k=0}^\infty a_{N,k}^n \, \alpha^{2k}, 
\end{equation}
where 
\begin{align}\label{a_{N,k}^n}
\begin{split}
	a_{N,k}^n&= \frac{ (-1)^{n} }{N^{k+1}\,2^{k-1}} \sum_{m=0}^{k} \frac{(2n+2m-1)!! \, (-1)^{k-m} }{m!\,(k-m)!\, (2n-2k+2m-1)!!}
	\\
	& \times \underset{\zeta,\eta=1}{\textup{Res}} \, \Big[ \frac{ \zeta^{-m-2n-1} \,\eta^{-k+m-1+2n} }{(\zeta-1)^{N-1-m-2n} (\eta-1)^{N-1-k+m}} \frac{ 1 }{ 1-(\zeta-1)^2 (\eta-1)^2 } \Big].
\end{split}
\end{align}
\end{lem}
\begin{proof}
We first express each summand in \eqref{INn} in terms of the residue of an elementary function. 
By the integral representation \eqref{2F1 int rep} of the hypergeometric function, for $x<0$, we have
\begin{align*}
&\quad \frac{{}_2\textbf{F}_1(l-m+\tfrac{1}{2},m-l+\tfrac12 ;-l-m+\tfrac{1}{2};x)}{(2m)!} 
\\
&=-\tfrac{1}{\Gamma(m-l+\tfrac12)} \, \underset{\zeta=1}{\text{Res}} \, \Big[ \zeta^{m-l-\frac12} (1-\zeta)^{-2m-1} (1-\zeta x)^{m-l-\frac12}\Big].
\end{align*}
This identity together with \eqref{Gamma reflection} leads to 
\begin{align*}
&\quad \frac{ {}_2\textbf{F}_1(n+\tfrac{1}{2},-n+\tfrac12 ;-2l+n+\tfrac{1}{2};x)^2 }{ (2l)! (2l-2n)!}
\\
&=\tfrac{(-1)^{n}}{\pi} \, \underset{\zeta,\eta=1}{\text{Res}} \, \Big[ \zeta^{-n-\frac12}\eta^{n-\frac12}\, (1-\zeta)^{-2l+2n-1}(1-\eta)^{-2l-1}\, (1-\zeta x)^{-n-\frac12} (1-\eta x)^{n-\frac12}\Big].
\end{align*}
We now simplify the sum of $(\zeta-1)^{-2l+2n-1} (\eta-1)^{-2l-1}$ and expand the factor $(1-\zeta x)^{-n-\frac12} 	(1-\eta x)^{n-\frac12}$ around $\alpha=0$ ($x=1-\frac{N}{\alpha^2}$).
As a rational function of $\zeta$ and $\eta$,
\begin{align*}
	\sum_{l=n}^{N/2-1} (\zeta-1)^{-2l+2n-1} (\eta-1)^{-2l-1}
\end{align*}
simplifies to 
\begin{align*}
\frac{(\zeta-1)^{2n-N+1}(\eta-1)^{-N+1}-(\zeta-1)(\eta-1)^{1-2n}}{1-(\zeta-1)^2 (\eta-1)^2}.
\end{align*}
Also for $x=1-\frac{N}{\alpha^2}$, the binomial expansions of $(1-\zeta x)^{-n-\frac12}$ and $(1-\eta x)^{n-\frac12}$ give that
\begin{align*}
	&\quad 	(1-\zeta x)^{-n-\frac12} 	(1-\eta x)^{n-\frac12}
		\\
	&=\sum_{k=0}^{\infty} \frac{ \alpha^{2k+2} }{N^{k+1}\,2^{k}} \sum_{m=0}^{k} \frac{(\zeta-1)^m}{ \zeta^{m+\frac12+n} } \frac{ (\eta-1)^{k-m} }{ \eta^{k-m+\frac12-n} } \frac{(2n+2m-1)!! \, (-1)^{k-m} }{m!\,(k-m)!\, (2n-2k+2m-1)!!}
\end{align*}
for $|\zeta-1| < |\zeta|\frac{N}{\alpha^2}$ and $|\eta-1| < |\eta|\frac{N}{\alpha^2}$. 
Combining all of the above equations with \eqref{INn}, the proof is complete. 
\end{proof}

Next, we compute the residue in \eqref{a_{N,k}^n}.

\begin{lem}\label{Lem_aNkn residue}
We have 
\begin{align}
	\begin{split} \label{a_{N,k}^n 2nd}
	a_{N,k}^n
	&= \frac{ (-1)^{k} }{2^{k-1}} \sum_{m=0}^{k} \frac{(2n+2m-1)!! \, (2k-2m-1-2n)!! }{m!\,(k-m)!\, (m+2n)!\, (k-m-2n)!}
	\\
	& \times \frac1{N^{k+1}}\sum_{l=0}^{N'} \frac{(N-2-2l)!(N-2-2l-2n)!}{ (N-2-2l-m-2n)!(N-2-2l-k+m)!},
		\end{split}
\end{align}
where $N':=\lfloor \frac{N-m}{2} \rfloor -n-1 .$
\end{lem}
\begin{proof}
Using the expansions 
\begin{align*}
&\zeta^{-m-2n-1}=\sum_{l=0}^{\infty} (-1)^l \frac{(m+2n+l)!}{(m+2n)! \, l!} (\zeta-1)^l, 
\\
&\frac{ 1 }{ 1-(\zeta-1)^2 (\eta-1)^2 }= \sum_{l=0}^{\infty} (\zeta-1)^{2l}(\eta-1)^{2l}
\end{align*}
around $\zeta=1$ and $\eta=1$,
we have
\begin{align*}
	&\quad \underset{\zeta=1}{\text{Res}} \, \Big[ \frac{ \zeta^{-m-2n-1} }{(\zeta-1)^{N-1-m-2n} } \frac{ 1 }{ 1-(\zeta-1)^2 (\eta-1)^2 } \Big]
	\\
	&=\frac{(-1)^{m}}{(m+2n)!}\sum_{l=0}^{N'} \frac{(N-2-2l)!}{ (N-2-m-2n-2l)!} (\eta-1)^{ 2l }.
\end{align*}
Therefore the residue in \eqref{a_{N,k}^n} is computed as 
\begin{align*}
	&\quad \underset{\zeta,\eta=1}{\text{Res}} \, \Big[ \frac{ \zeta^{-m-2n-1} \,\eta^{-k+m-1+2n} }{(\zeta-1)^{N-1-m-2n} (\eta-1)^{N-1-k+m}} \frac{ 1 }{ 1-(\zeta-1)^2 (\eta-1)^2 } \Big]
	\\
	&= \frac{(-1)^{m}}{(m+2n)!}\sum_{l=0}^{N'} \frac{(N-2-2l)!}{ (N-2-m-2n-2l)!} \, \underset{\eta=1}{\text{Res}} \Big[ \frac{ \eta^{-k+m-1+2n} }{(\eta-1)^{N-1-k+m-2l}} \Big].
	\\
&= \frac{(-1)^{k}}{(m+2n)!(k-m-2n)!}\sum_{l=0}^{N'} \frac{(N-2-2l)!(N-2-2l-2n)!}{ (N-2-2l-m-2n)!(N-2-2l-k+m)!},
\end{align*}
which completes the proof. 
\end{proof}

To combine each of the residues in \eqref{RN1N decomp}, we will need the following combinatorial identity.

\begin{lem} \label{Lem_Com iden}
For any non-negative integer $k$, we have 
\begin{equation} \label{Com iden}
\binom{2k}{k} = \sum_{m=0}^{k} \binom{k}{m} \sum_{ n=-\lfloor k/2 \rfloor }^{ \lfloor k/2 \rfloor } \frac{(2n+2m-1)!! \, (2k-2m-2n-1)!! }{\,(m+2n)!\, (k-m-2n)!}. 
\end{equation}
\end{lem}
\begin{proof}
Note that the inner summation of \eqref{Com iden} can be written as 
\begin{align*}
\frac{1}{2^k}\sideset{}{'}\sum_n \frac{(2m+2n)! \, (2k-2m-2n)! }{(m+n)!\,(k-m-n)! \,(m+2n)!\, (k-m-2n)!},
\end{align*}
where the summation $\sum'_n$ is taken over all $n$ such that $m+2n \ge 0$ and $k-m-2n \ge 0$. 
Then \eqref{Com iden} is equivalent to 
$$
2^k \binom{2k}{k}= \sum_{m=0}^{k} \sideset{}{'}\sum_n \binom{k}{m+n}\binom{2m+2n}{m} \binom{2k-2m-2n}{k-m}.
$$
This identity easily follows from the elementary combinatorics using the double-counting method. 
To be more precise, the left-hand side corresponds to the number of cases where $k$ out of $2k$ balls are chosen, and each is painted black or red (say). 
On the other hand, the right-hand side can be interpreted as the total number of cases in which $m +n$ pairs are first selected from the $k$ pairs of $2$ balls each, and $m$ balls are selected out of the selected $m+n$ pairs of balls and painted black, while $k-m$ balls are selected out of the rest $k-m-n$ pairs with $2k-2m-2n$ balls and painted red. This completes the proof. 

\end{proof}

Now we are ready to prove Theorem~\ref{Thm_Variance}.

\begin{proof}[Proof of Theorem~\ref{Thm_Variance}]
Recall that we aim to prove \eqref{VN Goal11}. 
By Lemma~\ref{Lem_SN1 int}, it is enough to show that 
\begin{equation}\label{VN Goal2}
	\lim_{ N \to \infty } \RN{1}_N(\alpha) = 	\lim_{ N \to \infty } \RN{2}_N(\alpha)= \frac12 \sum_{k=0}^\infty a_k \alpha^{2k}, \qquad a_k=(-1)^k\frac{(2k-1)!!}{k!(k+1)!} . 
\end{equation}
We shall present the proof of \eqref{VN Goal2} for $\RN{1}_N.$ The other one for $\RN{2}_N$ is left to the reader as an exercise.

Due to Lemma~\ref{Lem_INn power} and the decomposition \eqref{RN1N decomp}, the convergence \eqref{VN Goal2} is equivalent to
\begin{equation}\label{VN Goal4} 
a_k=a_{k}^0+2 \sum_{ n=1 }^{ \infty } a_k^{n}, \qquad a_{k}^n:=\lim_{ N \to \infty } a_{N,k}^n.
\end{equation}
To show \eqref{VN Goal4}, we first claim that 
\begin{equation} \label{a_{k}^n}
a_{k}^n= \frac{ (-1)^k }{2^{k}} \frac{1}{k+1} \sum_{m=0}^{k} \frac{(2n+2m-1)!! \, (2k-2m-1-2n)!! }{m!\,(k-m)!\, (m+2n)!\, (k-m-2n)!}.
\end{equation}
This follows from Lemma~\ref{Lem_aNkn residue} and Riemann sum approximation. More precisely, as $N \to \infty$,
the last factor in \eqref{a_{N,k}^n 2nd} has the following asymptotic behavior 
\begin{align*}
&\quad \frac{1}{N^{k+1}} \sum_{l=0}^{N'} \frac{(N-2l-2)!(N-2l-2n-2)!}{ (N-2l-m-2n-2)!(N-2l-k+m-2)!} 
\\
&= \frac{1}{N} \sum_{l=0}^{N'} \Big(1-\frac{2l}{N} + O\big(\frac{1}{N}\big)\Big)^k \sim \frac12 \int_0^1 (1-t)^k\,dt=\frac{1}{2(k+1)},
\end{align*}
which leads to \eqref{a_{k}^n}. 

Now observe that by \eqref{a_{N,k}^n 2nd}, we have 
$$a_{N,k}^n=a_{k}^n=0 \quad \text{if}\quad n> \lfloor k/2 \rfloor.$$
Then it follows from \eqref{a_{k}^n} and Lemma~\ref{Lem_Com iden} that 
\begin{align*}
a_{k}^0+2 \sum_{ n=1 }^{ \infty } a_k^{n}&= a_{k}^0+2 \sum_{ n=1 }^{ \lfloor k/2 \rfloor } a_k^{n}= (-1)^k\frac{(2k-1)!!}{k!(k+1)!}. 
\end{align*}
This completes the proof. 
\end{proof}

\bibliographystyle{abbrv}
\bibliography{RMTbib}

\begin{thebibliography}{10}

\bibitem{MR4229527}
G.~Akemann, S.-S. Byun, and N.-G. Kang.
\newblock A non-{H}ermitian generalisation of the {M}archenko-{P}astur
  distribution: From the circular law to multi-criticality.
\newblock {\em Ann. Henri Poincar\'{e}}, 22(4):1035--1068, 2021.

\bibitem{MR3845296}
G.~Akemann, M.~Cikovic, and M.~Venker.
\newblock Universality at weak and strong non-{H}ermiticity beyond the elliptic
  {G}inibre ensemble.
\newblock {\em Comm. Math. Phys.}, 362(3):1111--1141, 2018.

\bibitem{MR2679835}
G.~Akemann, M.~Kieburg, and M.~J. Phillips.
\newblock Skew-orthogonal {L}aguerre polynomials for chiral real asymmetric
  random matrices.
\newblock {\em J. Phys. A}, 43(37):375207, 24, 2010.

\bibitem{MR2592354}
G.~Akemann, M.~J. Phillips, and H.-J. Sommers.
\newblock The chiral {G}aussian two-matrix ensemble of real asymmetric
  matrices.
\newblock {\em J. Phys. A}, 43(8):085211, 29, 2010.

\bibitem{alt2021local}
J.~Alt and T.~Kr{\"u}ger.
\newblock Local elliptic law.
\newblock {\em preprint arXiv:2102.03335}, 2021.

\bibitem{AB20}
Y.~Ameur and S.-S. Byun.
\newblock Almost-{H}ermitian random matrices and bandlimited point processes.
\newblock {\em preprint arXiv:2101.03832}, 2021.

\bibitem{MR4068316}
J.~Baik and T.~Bothner.
\newblock The largest real eigenvalue in the real {G}inibre ensemble and its
  relation to the {Z}akharov-{S}habat system.
\newblock {\em Ann. Appl. Probab.}, 30(1):460--501, 2020.

\bibitem{MR2530159}
A.~Borodin and C.~D. Sinclair.
\newblock The {G}inibre ensemble of real random matrices and its scaling
  limits.
\newblock {\em Comm. Math. Phys.}, 291(1):177--224, 2009.

\bibitem{MR3230002}
P.~Bourgade, H.-T. Yau, and J.~Yin.
\newblock Local circular law for random matrices.
\newblock {\em Probab. Theory Related Fields}, 159(3-4):545--595, 2014.

\bibitem{MR1677884}
P.~A. Deift.
\newblock {\em Orthogonal polynomials and random matrices: a
  {R}iemann-{H}ilbert approach}, volume~3 of {\em Courant Lecture Notes in
  Mathematics}.
\newblock New York University, Courant Institute of Mathematical Sciences, New
  York; American Mathematical Society, Providence, RI, 1999.

\bibitem{MR2364955}
D.~Dominici.
\newblock Asymptotic analysis of the {H}ermite polynomials from their
  differential-difference equation.
\newblock {\em J. Difference Equ. Appl.}, 13(12):1115--1128, 2007.

\bibitem{MR1231689}
A.~Edelman, E.~Kostlan, and M.~Shub.
\newblock How many eigenvalues of a random matrix are real?
\newblock {\em J. Amer. Math. Soc.}, 7(1):247--267, 1994.

\bibitem{efetov1997directed}
K.~B. Efetov.
\newblock Directed quantum chaos.
\newblock {\em Phys. Rev. Lett.}, 79(3):491, 1997.

\bibitem{forrester2020many}
P.~J. Forrester, J.~R. Ipsen, and S.~Kumar.
\newblock How many eigenvalues of a product of truncated orthogonal matrices
  are real?
\newblock {\em Exp. Math.}, 29(3):276--290, 2020.

\bibitem{MR2485724}
P.~J. Forrester and A.~Mays.
\newblock A method to calculate correlation functions for {$\beta=1$} random
  matrices of odd size.
\newblock {\em J. Stat. Phys.}, 134(3):443--462, 2009.

\bibitem{forrester2007eigenvalue}
P.~J. Forrester and T.~Nagao.
\newblock Eigenvalue statistics of the real {G}inibre ensemble.
\newblock {\em Phys. Rev. Lett.}, 99(5):050603, 2007.

\bibitem{MR2430570}
P.~J. Forrester and T.~Nagao.
\newblock Skew orthogonal polynomials and the partly symmetric real {G}inibre
  ensemble.
\newblock {\em J. Phys. A}, 41(37):375003, 19, 2008.

\bibitem{fyodorov1997almost}
Y.~V. Fyodorov, B.~A. Khoruzhenko, and H.-J. Sommers.
\newblock Almost {H}ermitian random matrices: crossover from {W}igner-{D}yson
  to {G}inibre eigenvalue statistics.
\newblock {\em Phys. Rev. Lett.}, 79(4):557--560, 1997.

\bibitem{MR1431718}
Y.~V. Fyodorov, B.~A. Khoruzhenko, and H.-J. Sommers.
\newblock Almost-{H}ermitian random matrices: eigenvalue density in the complex
  plane.
\newblock {\em Phys. Lett. A}, 226(1-2):46--52, 1997.

\bibitem{MR1634312}
Y.~V. Fyodorov, H.-J. Sommers, and B.~A. Khoruzhenko.
\newblock Universality in the random matrix spectra in the regime of weak
  non-{H}ermiticity.
\newblock {\em Ann. Inst. H. Poincar\'{e} Phys. Th\'{e}or.}, 68(4):449--489,
  1998.

\bibitem{FT20}
Y.~V. Fyodorov and W.~Tarnowski.
\newblock Condition numbers for real eigenvalues in the real elliptic
  {G}aussian ensemble.
\newblock {\em Ann. Henri Poincar\'{e}}, 22(1):309--330, 2021.

\bibitem{ginibre1965statistical}
J.~Ginibre.
\newblock Statistical ensembles of complex, quaternion, and real matrices.
\newblock {\em J. Math. Phys.}, 6(3):440--449, 1965.

\bibitem{girko1986elliptic}
V.~L. Girko.
\newblock Elliptic law.
\newblock {\em Theory Probab. Appl.}, 30(4):677--690, 1986.

\bibitem{MR2663633}
F.~G\"{o}tze and A.~Tikhomirov.
\newblock The circular law for random matrices.
\newblock {\em Ann. Probab.}, 38(4):1444--1491, 2010.

\bibitem{gradshteyn2014table}
I.~S. Gradshteyn and I.~M. Ryzhik.
\newblock {\em Table of integrals, series, and products}.
\newblock Academic press, 2014.

\bibitem{MR2185860}
E.~Kanzieper and G.~Akemann.
\newblock Statistics of real eigenvalues in {G}inibre's ensemble of random real
  matrices.
\newblock {\em Phys. Rev. Lett.}, 95(23):230201, 4, 2005.

\bibitem{MR2788023}
B.~A. Khoruzhenko, H.-J. Sommers, and K.~\.{Z}yczkowski.
\newblock Truncations of random orthogonal matrices.
\newblock {\em Phys. Rev. E (3)}, 82(4):040106, 4, 2010.

\bibitem{lee2016fine}
S.-Y. Lee and R.~Riser.
\newblock Fine asymptotic behavior for eigenvalues of random normal matrices:
  {E}llipse case.
\newblock {\em J. Math. Phys.}, 57(2):023302, 2016.

\bibitem{little2021number}
A.~Little, F.~Mezzadri, and N.~Simm.
\newblock On the number of real eigenvalues of a product of truncated
  orthogonal random matrices.
\newblock {\em preprint arXiv:2102.08842}, 2021.

\bibitem{MR3403996}
H.~H. Nguyen and S.~O'Rourke.
\newblock The elliptic law.
\newblock {\em Int. Math. Res. Not. IMRN}, 2015(17):7620--7689, 2015.

\bibitem{olver2010nist}
F.~W. Olver, D.~W. Lozier, R.~F. Boisvert, and C.~W. Clark~(Editors).
\newblock {\em NIST Handbook of Mathematical Functions}.
\newblock Cambridge University Press, Cambridge, 2010.

\bibitem{MR3678474}
M.~Poplavskyi, R.~Tribe, and O.~Zaboronski.
\newblock On the distribution of the largest real eigenvalue for the real
  {G}inibre ensemble.
\newblock {\em Ann. Appl. Probab.}, 27(3):1395--1413, 2017.

\bibitem{MR3685239}
N.~Simm.
\newblock On the real spectrum of a product of {G}aussian matrices.
\newblock {\em Electron. Commun. Probab.}, 22:Paper No. 41, 11, 2017.

\bibitem{MR2341601}
C.~D. Sinclair.
\newblock Averages over {G}inibre's ensemble of random real matrices.
\newblock {\em Int. Math. Res. Not. IMRN}, (5):Art. ID rnm015, 15, 2007.

\bibitem{MR2371225}
H.-J. Sommers.
\newblock Symplectic structure of the real {G}inibre ensemble.
\newblock {\em J. Phys. A}, 40(29):F671--F676, 2007.

\bibitem{MR948613}
H.-J. Sommers, A.~Crisanti, H.~Sompolinsky, and Y.~Stein.
\newblock Spectrum of large random asymmetric matrices.
\newblock {\em Phys. Rev. Lett.}, 60(19):1895--1898, 1988.

\bibitem{MR2722794}
T.~Tao and V.~Vu.
\newblock Random matrices: universality of {ESD}s and the circular law.
\newblock {\em Ann. Probab.}, 38(5):2023--2065, 2010.
\newblock With an appendix by Manjunath Krishnapur.

\bibitem{MR3306005}
T.~Tao and V.~Vu.
\newblock Random matrices: universality of local spectral statistics of
  non-{H}ermitian matrices.
\newblock {\em Ann. Probab.}, 43(2):782--874, 2015.

\end{thebibliography}
\end{document}